\documentclass[11pt,twoside, a4paper, english, reqno]{amsart}
\usepackage[dvips]{epsfig}
\usepackage{amscd}
\usepackage{amssymb}
\usepackage{amsthm}
\usepackage{amsmath}
\usepackage{latexsym}
\usepackage{amsaddr} 
\usepackage{mathrsfs}
\usepackage{fancyhdr}
\usepackage{mathtools}
\usepackage{comment}

\usepackage{upref}
\usepackage{hyperref}
\hypersetup{linkcolor=blue, colorlinks=true
,citecolor = red}
\usepackage{color}
\theoremstyle{plain}

\newtheorem{thm}{Theorem}[section]
\theoremstyle{plain}
\newtheorem{lem}[thm]{Lemma}
\newtheorem{prop}[thm]{Proposition}
\newtheorem{cor}[thm]{Corollary}

\theoremstyle{definition}
\newtheorem{defi}{Definition}[section]
\newtheorem{rem}[thm]{Remark}

\newcommand{\rN}{\mathbb{R}^{N}}
\newcommand{\rk}{\mathbb{R}^k}
\newcommand{\rNk}{\mathbb{R}^{N-k}}
\newcommand{\ro}{\mathbb{R}}
\newcommand{\Om}{\Omega}

\newcommand{\mcal}[1]{\mathcal{#1}}
\newcommand{\md}[1]{\left|#1 \right|}
\newcommand{\nrm}[1]{\left|\left| #1 \right|\right|}
\newcommand{\bct}[1]{\left(#1\right)}

\newcommand{\ta}{\tilde{\alpha}}

\numberwithin{equation}{section} \allowdisplaybreaks

        \title{Extremals For Fractional Order Hardy-Sobolev-Maz'ya Inequality}
\author{Arka Mallick}
\address{TIFR  Centre for Applicable Mathematics, Post Bag No. 6503
 Sharadanagar, \\ Bangalore 560065, India.}
\email{arkamallick02@gmail.com, arka@math.tifrbng.res.in} 

\begin{document}

\maketitle
\begin{abstract}
In this article, we derive the existence of positive solution of a semi-linear, non-local  elliptic PDE, involving a singular perturbation of the fractional laplacian, coming from the fractional Hardy-Sobolev-Maz'ya inequality, derived in this paper. We also derive symmetry properties and a precise asymptotic behaviour of solutions.
\end{abstract}



\section{Introduction}
 In this article, we study the following equation
 \begin{align}\label{flap_cyl}
\begin{rcases}
\frac{2}{c_{N,s}}\bct{-\Delta}^s u - \beta \frac{u}{|x'|^{2s}} = \frac{u^{2^*_t - 1}}{|x'|^t}, \ \text{in } \rN, \\ u\geq0 , \ \ u\in  \mcal{\dot{H}}^s_{\beta}(\rN),
\end{rcases}
\end{align}
where $c_{N,s}= 2^{2s}\pi^{-N/2}\frac{\Gamma\bct{\frac{N+2s}{2}}}{|\Gamma(-s)|}$, $0<s<1$, $0\leq t<2s$, $2^*_t = \frac{2(N-t)}{N-2s}$, $\rN = \rk\times\rNk$, $2\leq k\leq N-2$ and a point $x\in \rN$ is denoted as $x= (x',x'')\in \rk\times \rNk$. The space $\mcal{\dot{H}}^s_\beta (\rN)$ is defined in Section \ref{Prem}, where $0<\beta \leq \mcal{C}_{N,k,s}$ and $\mcal{C}_{N,k,s}$ is the best constant of the following fractional version of Hardy-Maz'ya inequality

\begin{align}\label{FHM_ineq} 
\int_{\rN}\int_{\rN}\frac{|u(x)-u(y)|^2}{|x-y|^{N+2s}} dx dy \geq \mcal{C}_{N,k,s} \int_{\rN} \frac{u^2(x)}{|x'|^{2s}}dx ,\ \forall u\in C_c^{0,1} \left(\rk \setminus \{0\} \times \mathbb{R}^{N-k}\right).
\end{align}
An explicit expression of $\mcal{C}_{N,k,s}$ and the proof of \eqref{FHM_ineq} can be found in Section \ref{frac_ineq}. For $k=N$ the inequality \eqref{FHM_ineq} has been derived in \cite{FrLiSi}, \cite{FS1}, \cite{Tzi}. Also, see \cite{Ili61}, for similar inequalities.
\par The local counterpart (i.e. for the case $s=1$) of \eqref{FHM_ineq} is derived by Maz'ya in \cite{MV} which can be state as follows
\begin{align}\label{HM_ineq}
\int_{\rN}\md{\nabla u(x)}^2 dx \geq \bct{\frac{k-2}{2}}^2 \int_{\rN} \frac{u^2(x)}{|x'|^2}dx, \ \forall u\in C_c^\infty \left(\rk \setminus \{0\} \times \mathbb{R}^{N-k}\right), 
\end{align}
where the constant $\frac{(k-2)^2}{4}$ is the best possible. When $k=N$, \eqref{HM_ineq} reduces to the usual Hardy inequality. Unlike the case of Hardy inequality, it was Maz'ya \cite{MV}, who first observed that, the following Sobolev type improvement of \eqref{HM_ineq} can be achieved when $k\geq 2$
\begin{align}\label{HSM_ineq}
C_\beta^t \bct{\int_{\rN}\frac{\md{u(x)}^{q}}{\md{x'}^t}dx}^{\frac{2}{q}} \leq \int_{\rN}\md{\nabla u(x)}^2 dx- \beta \int_{\rN} \frac{u^2(x)}{|x'|^2}dx, 
\end{align}
where $\beta \leq \frac{(k-2)^2}{4}$, $0\leq t <2$, $q= \frac{2(N-t)}{N-2}$ and $u\in C_c^\infty \left(\rk \setminus \{0\} \times \mathbb{R}^{N-k}\right)$. Note that, existence of nontrivial solution of \eqref{flap_cyl} will follow for the case of $s=1$, if we can show the existence of minimizers of \eqref{HSM_ineq}. For $\beta=0$, the existence of minimizers of \eqref{HSM_ineq} has been established in \cite{BadTar} by using concentration compactness principle due to P.L Lions \cite{PLL1}, \cite{PLL2}.  Whereas, for $0<\beta<\frac{(k-2)^2}{4}$, the existence is proved in \cite{musina}, by using blow up analysis for approximate solutions with a rescaling argument. On the other hand, since for $\beta = \frac{(k-2)^2}{4}$, the expected space in which the minimizers will belong is much bigger than the same for the case of $\beta<\frac{(k-2)^2}{4}$, one needs to employ a careful analysis. Using a penalty method, Tintarev and Tertikas proved the existence of minimizers for the case of $\beta = \frac{(k-2)^2}{4}$ in \cite{TinTer} and subsequently improved in \cite{MusGaz1}.
\par The cylindrical symmetry of the local counterpart (i.e. $s=1$) of \eqref{flap_cyl}, has been established in \cite{SanManFab} by using moving plane method in the special case of $\beta=0$. In fact, when $t=1$, they have classified all the solutions by a careful asymptotic analysis. Subsequently, in the case of $0\leq\beta<\frac{(k-2)^2}{4}$, $0\leq t<2$ and $s=1$, cylindrical symmetry of solutions of \eqref{flap_cyl} has been established in \cite{MusGaz}. Finally, in a breakthrough paper, Sandeep and Mancini \cite{SanMan}, established the uniqueness of positive extremals of \eqref{HSM_ineq}. See also \cite{CasFabManSan}.

\par Thus we need the following fractional version of Hardy-Sobolev-Mazya inequality, to prove the existence of solution of \eqref{flap_cyl}. 

\begin{align}\label{FHSM_ineq}
C\bct{\int_{\rN}\frac{\md{u(x)}^{2^*_t}}{\md{x'}^t}dx}^{\frac{2}{2^*_t}} \leq \int_{\rN}\int_{\rN}\frac{|u(x)-u(y)|^2}{|x-y|^{N+2s}} dx dy - \beta \int_{\rN} \frac{u^2(x)}{|x'|^{2s}}dx,
\end{align}
where $0\leq t <2s$, $2^*_t:= \frac{2(N-t)}{N-2s}$, $0\leq \beta \leq \mcal{C}_{N,k,s}$, $\mcal{C}_{N,k,s}$ is the optimal constant of \eqref{FHM_ineq} and $C>0$ is a constant independent of $u$. 
We will establish the inequality \eqref{FHSM_ineq} in Section \ref{frac_ineq}. In Section \ref{exis}, we have proved the following theorems to establish the existence of a solution to the equation \eqref{flap_cyl}.

\begin{thm}\label{exis_FHSM}
Let $0<s<1$, $0\leq t<2s$ and $0\leq\beta \leq \mcal{C}_{N,k,s}$. Define $S^t_k(\beta)$ in the following manner

\begin{align}\label{bst_const1}
S^t_k(\beta) := \displaystyle \sup_{\substack{u\in C_c^{0,1} \left(\rk \setminus \{0\} \times \mathbb{R}^{N-k}\right),\\ u\neq 0}} \frac{\int_{\rN}\frac{\left|u\right|^{2^*_t}}{|x'|^t}}{\left(\int_{\rN}\int_{\rN}\frac{|u(x)-u(y)|^2}{|x-y|^{N+2s}} dx dy - \beta \int_{\rN} \frac{u^2(x)}{|x'|^{2s}}dx\right) ^{\frac{2^*_t}{2}}}.
\end{align}
Then, the supremum is achieved in $\mcal{\dot{H}}_\beta^s(\rN)$, if $0<\beta<\mcal{C}_{N,k,s},$ where the space $\mcal{\dot{H}}_\beta^s(\rN)$ is defined in Section \ref{Prem}.
\end{thm}

\begin{thm}\label{exisb_FHSM}
The following infimum
\begin{align}\label{bst_const2}
    \kappa_k^t := \displaystyle \inf_{\substack{u\in \mcal{\dot{H}}^{s,\alpha}(\rN),\\ u\neq 0}} = \frac{\int_{\rN}\int_{\rN}\frac{|u(x)-u(y)|^2}{|x-y|^{N+2s} |x'|^\alpha|y'|^\alpha}dxdy}{\bct{\int_{\rN}\frac{|u|^{2^*_t}}{\left|x'\right|^{t+2_t^*\alpha}}dx}^{\frac{2}{2^*_t}}} 
\end{align}
 is achieved in $\mcal{\dot{H}}^{s,\alpha}(\rN)$. Here, $2^*_t = \frac{2(N-t)}{N-2s}$, $0\leq t <2s$, and $\alpha:= (k-2s)/2$ and the space $\mcal{\dot{H}}^{s,\alpha}(\rN)$ is defined in the Definition \ref{spaces}.
\end{thm}

We remark that, Theorem \ref{exis_FHSM} settles the issue of existence of solution of \eqref{flap_cyl}, for $0<\beta<\mcal{C}_{N,k,s}$. Whereas, using a combination of Theorem \ref{exisb_FHSM} and Theorem \ref{FHM_grndst},  we can conclude that there exists a nontrivial solution of \eqref{flap_cyl}, for the case $\beta = \mcal{C}_{N,k,s}$. To prove Theorem \ref{exis_FHSM}, we have used an improved  version of fractional Sobolev inequality  derived in Proposition \ref{im_sob} originated in \cite{PP}. On the other hand, as pointed out in Section \ref{crit_space}, the space $\mcal{\dot{H}}^s_{\mcal{C}_{N,k,s}}(\rN)$ is  much bigger than the space $\mcal{\dot{H}}^s_\beta(\rN)$, which is nothing but the usual homogeneous fractional Sobolev space $\dot{H}^s(\rN)$, in the case of $0<\beta<\mcal{C}_{N,k,s}$. Because of this fact, we could not use Proposition \ref{im_sob}. We also note that, in this case, we cannot use extension method, derived in \cite{CS}, to convert \eqref{flap_cyl} into a local one. Rather, by using Ekeland's variation principle, we were able to conclude that, the approximate solutions cannot concentrate near the singular set. However, since we are in a non local setup, so we faced a natural difficulty when we tried to cut off the approximate solutions. In this context, let us mention a paper by Ghoussoub, and Shakerian \cite{GhSh}, where they used Ekeland's variation principle to get the existence of solution of a nonlocal equation, similar to \eqref{flap_cyl}. But, their arguments were based on extension technique. For a comprehensive study for the case $k=1$, see \cite{MusNaz}.

Next, we prove some qualitative properties of solutions of \eqref{flap_cyl}  by means of following theorems.

\begin{thm}\label{reg_flap_cyl}
Let $u$ satisfies \eqref{flap_cyl} weakly. Then, for any $0<\beta \leq \mcal{C}_{N,k,s}$, there exists a unique $\ta\in (0,\frac{k-2s}{2}]$ given by \eqref{alphatilde}, such that the following holds:
\begin{align}\label{asy_est}
0\leq u(x) \leq \frac{C}{|x'|^{\tilde{\alpha}}\bct{1+\md{x}^{N-2s-2\tilde{\alpha}}}} ,  \ \forall x\in \rN_k,
\end{align}
where $C>0$ is constant, depends on $u$ but independent of $x\in \rN_k$ and \\ $\rN_k:= \bct{\rk\setminus \{0\}}\times \rNk$. Moreover, if $0<\beta<\mcal{C}_{N,k,s}$, then $u\in C^\infty(\rN_k)$.  
\end{thm}

\begin{thm}\label{cyl_sym}
For $0<\beta \leq \mcal{C}_{N,k,s} $ and $0<s<1$, any $u\in C^{0,1}_{loc}(\rN_k)$ satisfying \eqref{flap_cyl} is cylindrically symmetric i.e. radial with respect to $x' \in \rk$ and there exist $x''_0\in \mathbb{R}^{N-k}$ such that for any fixed $x'\neq 0$, $u(x',x'') $ is radial in the second variable with respect to $x''_0.$
\end{thm}
We have used moving plane method to prove Theorem \ref{cyl_sym}. Along with other hurdles, applying moving plane method in the non-local setup is inconvenient due to inadequacy of any direct small measure type lemma which was observed, in the local case, by Varadhan and successfully disseminated by Nirenberg and Berestycki \cite{NB}. As observed in \cite{SanManFab}, to prove such small measure type lemma, we could try to use the test function $w_\lambda =(u-u_\lambda)^+$, where $u_\lambda$ denotes the usual reflexion of $u$ along the hyperplane $T_\lambda= \{x_{i}= \lambda\}$. But since we are in non local set up, as noticed in \cite{DMPB}, \cite{FelWan}, the right test function should be an odd reflexion of $w_\lambda$ along $T_\lambda$. However, in our set up we faced difficulty in showing $w_\lambda$ belongs to right space simply because of the following reason. When $0<\beta<\mcal{C}_{N,k,s}$, to show the odd extension of $w_\lambda$ belongs to $\dot{H}^s(\rN)$, we need to show that
\begin{align*} 
\int_{\{x_i<\lambda\}} \frac{w_\lambda^2(x)}{(\lambda-x_i)^{2s}}dx <\infty,
\end{align*}
which follows by a fractional Hardy inequality proved in \cite{BogDyd}, when $s\neq \frac{1}{2}$ and $0<s<1$. But, when $s=\frac{1}{2}$, the best constant of the fractional Hardy inequality is zero. On the other hand, when $\beta = \mcal{C}_{N,k,s}$, due to unavailability of right Hardy type inequality, the situation becomes more complex and it is not clear whether we can use the odd reflexion of $w_\lambda$ as a test function or not. We have avoided this difficulty by approximating $w_\lambda$ properly and using the precise bound on $u$, derived in Theorem \ref{reg_flap_cyl}. 
 
 Finally, let us describe the plan of this article. In Section \ref{Prem}, we 
 will introduce the notations and all the function spaces used in this article. We will also recall some of the known results used in the proofs. Section \ref{frac_ineq} will be devoted to the proofs of inequality \eqref{FHM_ineq} and \eqref{FHSM_ineq}. Section \ref{exis} contains the proofs of Theorem \ref{exis_FHSM} and \ref{exisb_FHSM}. Also, Section \ref{reg} and \ref{symm} contains the proofs of Theorem \ref{reg_flap_cyl} and \ref{cyl_sym} respectively. Finally, in the Appendix we have proved the Lemma \ref{rep_sob_sp}.
\section{Notations and Preliminaries}\label{Prem}

\textbf{Notations:} 
 We will denote the projection of a point $x\in \rN$ to $\rk$ and $\rNk$ by $x'$ and $x''$ respectively. \\ 
For any $l \in \mathbb{N}$ and any $z\in \mathbb{R}^l$ we denote the $l$ dimensional ball of radius $R>0$ centered at $z$ by $B^l_R(z).$\\
For $2\leq k \leq N-2$, $\rN_k$ stands for the set  $\bct{\rk\setminus\{0\}} \times \rNk$.\\ To avoid confusion, we clarify that  $2^* := \frac{2N}{N-2s}$ and $\alpha := \frac{k-2s}{2}$. \\

\subsection{Definitions and Different Notions Of Solution}
\par In this section we will define different function spaces to be used. We  will also define different notions of solution. 
\begin{defi}\label{spaces}

\begin{itemize}
\noindent
\item[(i)] We define $\mcal{\dot{H}}_{\beta}^{s}(\rN)$ as the completion of $C_c^{0,1} \left(\rk \setminus \{0\} \times \mathbb{R}^{N-k}\right)$ under the following norm
\begin{align*}
\left[u\right]^2_{s,\beta,\rN} : = \int_{\rN}\int_{\rN}\frac{|u(x)-u(y)|^2}{|x-y|^{N+2s}} dx dy - \beta \int_{\rN} \frac{u^2(x)}{|x'|^{2s}}dx. 
\end{align*}

Here $0\leq \beta \leq \mcal{C}_{N,k,s}$ and $\mcal{C}_{N,k,s}$ is the best constant of the Fractional Hardy-Sobolev-Maz'ya inequality.

\item[(ii)] We also define the space $\mcal{\dot{H}}^{s,\tilde{\alpha}}(\rN)$ as the completion of $C_c^{0,1} \left(\rk \setminus \{0\} \times \mathbb{R}^{N-k}\right)$ under the following norm

\begin{align*}
\left[\left[u\right]\right]^2_{s,\tilde{\alpha},\rN} :=  \int_{\rN}\int_{\rN} \frac{|u(x)-u(y)|^2}{|x-y|^{N+2s} \left|x'\right|^{\tilde{\alpha}} \left|y'\right|^{\tilde{\alpha}}} dx dy,
\end{align*}
where $0<\tilde{\alpha}\leq \alpha:=(k-2s)/2$.
\item[(iii)]We recall, $\dot{H}^s(\rN)$  is the completion of $C_c^\infty(\rN)$ under the following norm 

\begin{align*}
\left[u\right]^2_{s,\rN} := \int_{\rN}\int_{\rN}\frac{\md{u(x)-u(y)}^2}{\md{x-y}^{N+2s}}dx dy. 
\end{align*}
One can easily prove that the following characterization of $\dot{H}^s(\rN).$ See \cite{DV} 
\begin{align*}
\dot{H}^s(\rN) := \{u\in L^{2^*}(\rN): [u]_{s,\rN} <\infty \}. 
\end{align*}

\item[(iv)] For any domain $\Om \subset \mathbb{R}^{N+1}_+$ we recall 
\begin{align*}
H^1(\Om, 1-2s) := \{U\in L^2(\Om,1-2s): \md{\nabla U}\in L^2(\Om,1-2s)\},
\end{align*}
where, $L^2(\Om,1-2s) : = \{U:\Om \rightarrow \ro \ \text{measurable }| \int_{\Om}t^{1-2s}U^2(x,t) dxdt < \infty \}$
\item[(v)] We recall, $ L_s(\rN): = \left\{f:\rN\rightarrow \ro\text{ measurable }|\int_{\rN}\frac{|f|}{1+|x|^{N+2s}}dx <\infty\right\}$
\end{itemize}

\end{defi}

Finally, let us state the following lemma regarding the precise representation of $\mcal{\dot{H}}^{s,\tilde{\alpha}}(\rN)$.

\begin{lem}\label{rep_sob_sp}
For $0<s<1$, $k\geq 2$ and $0<\tilde{\alpha}\leq \alpha$, we have the following representation of $\mcal{\dot{H}}^{s,\tilde{\alpha}}(\rN)$.  
\begin{align*}
\mcal{\dot{H}}^{s,\tilde{\alpha}}(\rN) = \{u\in L^{2^*}\left(\rN;\frac{1}{|x'|^{\tilde{\alpha}2^*}}\right) : \int_{\rN}\int_{\rN}\frac{\md{u(x)-u(y)}^2 dxdy}{\md{x-y}^{N+2s}|x'|^{\tilde{\alpha}}|y'|^{\tilde{\alpha}}}< \infty \},
\end{align*}
where $L^{2^*}\left(\rN;\frac{1}{|x'|^{\tilde{\alpha}2^*}}\right)$ consists of all measurable function $f$ such that $\frac{f}{|x'|^{\tilde{\alpha}} }\in L^{2^*}(\rN)$.
\end{lem}

We have proved the lemma in the Appendix. We recall, for $f$ belonging to the Schwartz class the fractional laplacian can be defined by the following integral representation. See \cite{NPV}
\begin{align*}
\bct{-\Delta}^s f (x):= c_{N,s}P.V. \int_{\rN}\frac{f(x)-f(y)}{|x-y|^{N+2s}} dy, \text{ for } x\in \rN,
\end{align*}
where $c_{N,s}= 2^{2s}\pi^{-N/2}\frac{\Gamma\bct{\frac{N+2s}{2}}}{|\Gamma(-s)|}$. Next, we need the following two definitions.
\begin{defi}\label{wk_sol}
For $0<\beta \leq \mcal{C}_{N,k,s}$, we say that, $u\in \mcal{\dot{H}}^s_\beta(\rN)$ is a weak solution of  \eqref{flap_cyl} if for every $\psi \in \mcal{\dot{H}}^s_\beta(\rN)$, we have 
\begin{align*}
\int_{\rN}\int_{\rN} \frac{\bct{u(x)-u(y)}\bct{\psi(x)-\psi(y)}}{|x-y|^{N+2s}} dxdy - \beta \int_{\rN} \frac{u(x)\psi(x)}{|x'|^{2s}}dx = \int_{\rN}\frac{u^{2^*_t-1}(x)\psi(x)}{|x'|^t}dx
\end{align*}
\end{defi}

\begin{defi}
Let $\Om\subset \rN$ be any open set and $\mcal{D}'(\Om)$ denotes the space of all distributions over $\Om$. Assume that, $u\in L_s(\rN)$ and $f\in D'(\Om)$. Then we say 
\begin{align*}
\bct{-\Delta}^s u (\geq) = (\leq) f , \text{ in the distributional sense}, 
\end{align*}
if for any nonnegative $\phi \in C_c^{\infty}(\Om)$ we have
\begin{align*}
\int_{\rN} u(x)\bct{-\Delta}^s\phi(x)dx  (\geq) = (\leq ) \langle f,\phi \rangle_\Om,
\end{align*}
 $\langle f,\phi \rangle_\Om$ denotes the action of $f$ on $\phi$.
\end{defi}

\subsection{Some Known Results}
In this section, we will recall some known results.\\

\par \textbf{Master Inequality:} We recall the following integral version of P{\'o}lya-Szeg{\"o} inequality. See \cite{BA2}.
\begin{thm}\label{mas_inq}
Let $f,g \in C(\ro^{N})$ be non negative vanishing at infinity i.e. they satisfy, $\md{\{f>t\}}_N<\infty,$ $\forall t > \operatorname{ess} \inf f$ and $\md{\{g>t\}}_N<\infty,$ $\forall t > \operatorname{ess} \inf g$, where for any measurable subset $A$ of $\rN$, $\md{A}_N$ denotes the $N$-dimensional Lebesgue measure of $A.$ Then, for any fixed  $\phi: \ro^+\rightarrow \ro^+$ increasing, convex and $K: \ro^+\rightarrow \ro^+$ decreasing, we have the following inequality
\begin{align}\label{fpol_sze}
\int_{\rN}\int_{\rN} \phi\bct{\md{f^*(x)-g^*(y)}}K\bct{|x-y|} dx dy  \leq \int_{\rN}\int_{\rN} \phi\bct{\md{f(x)-g(y)}}K\bct{|x-y|} dx dy,
\end{align}
where $f^*$ and $g^*$ denotes the Schwarz symmetrization of $f,g$ respectively.
\end{thm}

We need the following improved version of Sobolev inequality.
\begin{prop}\label{im_sob}
Let $0\leq t<2s$ and $u\in \dot{H}(\rN).$ Then, there exist constants $C,\theta_1,\theta_2>0$ independent of $u$ such that 
 \begin{align}\label{im_sob2}
\int_{\rN} \frac{\md{u}^{2_t^*}}{|x'|^t} dx \leq C \left[u\right]_{s,\rN}^{\theta_1} \nrm{u}_{\mcal{L}^{2,N-2s}}^{\theta_2},
\end{align}
where 
\begin{align*}
\nrm{u}^2_{\mcal{L}^{2,N-2s}} := \displaystyle\sup_{R>0,x\in \rN} \frac{R^{N-2s}}{\md{B^N_R(x)}}\int_{B^N_R(x)}\md{u}^2 dy.
\end{align*}
\end{prop}
\begin{proof}
We divide the proof in the following two cases.
\par \textit{Case 1: } $t=0$ In this case, \eqref{im_sob2} was proved by Palatucci and Pisante in \cite{PP}. Their inequality states as follows: there exists a constant $C>0$ such that for any $u \in \dot{H}^s(\rN)$
\begin{align}\label{im_sob1}
\nrm{u}_{L^{2^*}(\rN)} \leq C \left[u\right]_{s,\rN}^\theta \nrm{u}^{1-\theta}_{\mcal{L}^{2,N-2s}}
\end{align}
where $\frac{2}{2^*} \leq \theta <1.$
\par \textit{Case 2: } $0<t<2s.$ Using H\"older inequality we get 
\begin{align}\label{im_sob3}
\int_{\rN} \frac{\md{u}^{2_t^*}}{\md{x'}^t} dx\leq \bct{\int_{\rN}\md{u}^{2^*}}^{\frac{2_t^*-2}{2^*}} \bct{\int_{\rN}\frac{\md{u}^\frac{22^*}{2*-2_t^*+2}}{\md{x'}^{\frac{2^*t}{2^*-2_t^*+2}}}}^{\frac{2^*-2_t^*+2}{2^*}}.
\end{align}
Now let $\xi := \frac{2^*t}{2^*-2_t^*+2}.$ Then clearly $0<\xi<2s$ and $2^*_\eta = \frac{2^*2}{2^*-2_t^*+2}.$ Hence using \eqref{im_sob1}, \eqref{im_sob3} and \eqref{FHSM_ineq} we get the desired inequality \eqref{im_sob2} with $\theta_1 = \bct{2^*_t - 2}\theta+ \frac{2^*}{2}\bct{1-\frac{2^*_t-2}{2^*}}$ and $ \theta_2 = (1-\theta)\bct{2^*_t-2},$ where $\frac{2^*}{2}<\theta <1$ same as in the inequality \eqref{im_sob1}. 
\end{proof}

\section{fractional Hardy-Maz'ya and Hardy-Sobolev-Maz'ya Inequality}\label{frac_ineq}

In this section, we will give two proofs of fractional Hardy-Maz'ya inequality \eqref{FHM_ineq}. While in the first method, we get \eqref{FHM_ineq} with best possible constant, in the second method, we get the inequality with a rough constant. Also, as a consequence of the results, derived using both the methods we will be able to prove \eqref{FHSM_ineq}.
\subsection{Ground State Representation and the Fractional Hardy-Maz'ya inequality}
In this subsection, we will derive an appropriate ground state representation. Similar representation was proved in \cite{FS1} to derive the fractional Hardy inequality.  In fact, we will use a few results derived in \cite{FS1}. For reader's convenience, let us recall their result. Before that, we need to recall the following assumption. 
\par \textbf{Assumption 1:} Let $\Om\subset \rN$ be any open set. We also assume that $w$ is an almost everywhere positive measurable function in $\Om$ and  there exists a family of measurable function $k_\epsilon$, $\epsilon>0$ on $\Om \times \Om $ satisfying  $k_\epsilon(x,y)= k_\epsilon (y,x)$, $0\leq k_\epsilon(x,y)\leq  k(x,y)$, and 
\begin{align*}
\lim_{\epsilon \rightarrow 0}k_{\epsilon}(x,y) = k(x,y), \ \text{for a.e. } x,y\in \Om. 
\end{align*}
 Moreover, the integrals
 \begin{align*}
 V_{\epsilon}(x):= \frac{2}{w(x)} \int_{\Om} \bct{w(x)-w(y)} k_\epsilon(x,y) dy
 \end{align*}
 are absolutely convergent for a.e. $x \in \Om$ and belong to $L^1_{loc}(\Om)$. In addition to this, we assume that $V:= \displaystyle\lim_{\epsilon \rightarrow 0}V_{\epsilon}$ exists weakly in $L^1_{loc}(\Om).$
 \begin{prop}[Frank and Seiringer]\label{grndst_gen}
 Under the Assumption 1, for any $u\in C_c^{\infty}(\Om)$ we write $v: = \frac{u}{w}$ and assume
 \begin{align*}
 E[u] &: = \int_{\Om}\int_{\Om} \md{u(x)-u(y)}^2 k(x,y) dy dx <\infty,\\
 E_w[v] &: = \int_{\Om}\int_{\Om} \md{v(x)-v(y)}^2 w(x)w(y)k(x,y) dy dx <\infty \ \text{and }\\ 
 \int_{\Om}V^+ |u|^2 dx &<\infty.
\end{align*}
Then $E[u]- \int_{\rN} V(x) \md{u(x)}^2 dx = E_w[v].$

 \end{prop}
 As a particular example, we take $\Om = (\rk \setminus \{0\}) \times \rNk$ , $w(x):= 1/|x'|^{\tilde{\alpha}}$, for \\  $0<\tilde{\alpha}\leq \frac{k-2s}{2}$, $k(x,y) : = \frac{1}{\md{x-y}^{N+2s}}$, $V(x) : = \mcal{C}_{N,k,s}(\tilde{\alpha}) \frac{1}{\md{x'}^{2s}},$ where  $\mcal{C}_{N,k,s}(\ta)$ is defined by \\ $\mcal{C}_{N,k,s}(\tilde{\alpha}) := \mcal{\overline{C}}_{k,s}(\tilde{\alpha}) \int_{\rNk}\frac{dy''}{\bct{1+\md{y''}^2}^\frac{N+2s}{2}}$ and the function $\mcal{\overline{C}}_{k,s}(\tilde{\alpha})$ is defined in Lemma \ref{grndst_par2} below. We denote $\mcal{C}_{N,k,s}:= \mcal{C}_{N,k,s}\bct{(k-2s)/2}$. We also take, $k_\epsilon(x,y) :=\frac{1}{\md{x-y}^{N+2s}} \chi _{\{\md{|x'|-|y'|}>\epsilon\}}$. Then we have the following lemma.
 \begin{lem}\label{grndst_par1}
 The following limit converges uniformly for $x$ from compact sets in $\rN_k$.
 \begin{align*}
 2\lim_{\epsilon \rightarrow 0} \int_{\rN} \bct{w(x)-w(y)}k_{\epsilon}(x,y) dy = \mcal{C}_{N,k,s} (\tilde{\alpha})\frac{w(x)}{\md{x'}^{2s}}.
 \end{align*}
 \end{lem}
\begin{proof}
It is enough to notice that the following identity is true. The rest will follow from Lemma \ref{grndst_par2} below.
\begin{align*}
2\int_{\rN_k} \bct{w(x)-w(y)}k_{\epsilon}(x,y) dy = I_{N,k,s} \int_{\md{|x'|-|y'|}>\epsilon} \frac{w(x')-w(y')}{\md{x'-y'}^{k+2s}} dy,
\end{align*}
where $I_{N,k,s} := \int_{\rNk}\frac{dy''}{\bct{1+\md{y''}^2}^\frac{N+2s}{2}}<\infty$. This proves the lemma.
\end{proof}

\begin{lem}\label{grndst_par2}
One has uniformly for $x$ from compact sets in $\rk \setminus \{0\}$
\begin{align*}
2\lim_{\epsilon \rightarrow 0} \int_{\rk} \frac{\bct{w(x')-w(y')}}{\md{x'-y'}^{k+2s}} dy = \mcal{\overline{C}}_{k,s}(\tilde{\alpha}) \frac{w(x')}{\md{x'}^{2s}},
\end{align*}
where 
\begin{align}\label{alphatilde1}
\mcal{\overline{C}}_{k,s}(\tilde{\alpha}) &:= 2\int_0^1 r^{2s-1}\md{1-r^{\tilde{\alpha}}}^2 \Phi_{k,s}(r) dr, \\ 
\Phi_{k,s}(r) &: = \md{S^{k-2}} \int_{-1}^1 \frac{\bct{1-t^2}^{\frac{k-3}{2}}dt}{\bct{1-2rt+r^2}^{\frac{k+2s}{2}}}, \ 0<\tilde{\alpha}\leq \frac{k-2s}{2}, \ \text{and } k\geq2. \notag
\end{align}

\end{lem}

The above lemma was proved in \cite{FS1} (Lemma 3.1), in the case of $\ta = (k-2s)/2$. But, it is easy to see that the same proof will work even for $0<\ta<\frac{k-2s}{2}$. Notice that 

\begin{align*}
\mcal{\overline{C}}_{k,s}(0)=0, \ \mcal{\overline{C}}_{k,s}(\frac{k-2s}{2})= 2\pi^{k/2}\frac{\Gamma^2\bct{(k+2s)/4}}{\Gamma^2\bct{(k-2s)/4}}\frac{\md{\Gamma(-s)}}{\Gamma\bct{(k+2s)/2}}
\end{align*}
 i.e. the best constant of fractional Hardy inequality in dimension $k$. Also, $\mcal{\overline{C}}_{k,s}(\tilde{\alpha})$ is strictly increasing and continuous in $[0,(k-2s)/2]$. So, for any $0\leq \beta \leq \mcal{C}_{N,k,s}$ there exists unique $\ta \in [0,\frac{k-2s}{2}]$ such that 
\begin{align}\label{alphatilde}
\beta = \mcal{C}_{N,k,s}(\ta) = \overline{\mcal{C}}_{k,\alpha}(\ta) \int_{\rNk} \frac{dy''}{\bct{1+|y''|^2}^{\frac{N+2s}{2}}},
\end{align}
where $\overline{\mcal{C}}_{k,s}(\ta)$ is defined in \eqref{alphatilde1}.
 Hence, summarizing the above discussion we have the following theorem. 
\begin{thm}[Ground State Represetation]\label{FHM_grndst}
Let $u \in C_c^{0,1}(\rk\setminus\{0\} \times \rNk)$, $k\geq 2$, and $0<s<1$. Then for any $0<\beta \leq \mcal{C}_{N,k,s}$ there exist a unique $0<\tilde{\alpha}\leq (k-2s)/2$ such that 
\begin{align}\label{grndst_rep}
\int_{\rN}\int_{\rN} \frac{\md{u(x)-u(y)}^2}{\md{x-y}^{N+2s}}dx dy -\beta \int_{\rN} \frac{u^2(x)}{\md{x'}^{2s}}dx = \int_{\rN}\int_{\rN} \frac{\md{v(x)-v(y)}^2}{\md{x-y}^{N+2s}\md{x'}^{\tilde{\alpha}}\md{y'}^{\tilde{\alpha}}}dx dy.
\end{align}
 Here, $v(x) = \md{x'}^{\tilde{\alpha}}u(x)$, $\mcal{C}_{N,k,s} = \mcal{\bar{C}}_{k,s} \int_{\rNk}\frac{dy''}{\bct{1+\md{y''}^2}^\frac{N+2s}{2}}$ and $\mcal{\bar{C}}_{k,s}$ is the best constant of the fractional Hardy inequality for dimension $k.$ Moreover, $\tilde{\alpha}= (k-2s)/2$ whenever $\beta = \mcal{C}_{N,k,s}.$
\end{thm}
As a consequence of the above ground state representation, we have the fractional Hardy-Maz'ya inequality \eqref{FHM_ineq}. Also, one can follow the same lines of Frank and Seiringer \cite{FS2}, to conclude that $\mcal{C}_{N,k,s}$ is actually the best constant.

\subsection{John Domain and Fractional Hardy-Sobolev-Maz'ya inequality} 
 To prove the fractional Hardy-Sobolev-Maz'ya inequality,  we will use the fact,  $\rN_k$ is a John domain for $2\leq k\leq N-2$ and a recent result by Dyda, Lehrb\"ack and V\"ah\"akangas \cite{DLV}. For reader's convenience we will state their result below. First, let us recall some definitions.
 \begin{defi}[Assouad dimension]
 For $D \subset \rN$, the Assouad dimension denoted by $\dim_A(D)$ is the infimum of all exponent $\beta \geq 0$, for which there is a constant $C\geq1$, such that for every $x\in D$ and every $0<r<R$, the set $D\cap B^N_R(x)$ can be covered by at most $C\bct{\frac{R}{r}}^\beta$  balls of radius $r$.
 \end{defi} 
\begin{defi}[John Domain]
A domain $\Om \subsetneq \rN$, with $N\geq 2$, is a c-John domain, for $c\geq1,$ if each pair of points $x_1,x_2 \in \Om$ can be joined by a rectifiable arc length parametrized curve $\gamma : [0,l]\rightarrow \Om$ satisfying $\operatorname{dist}\bct{\gamma(t),\Om}\geq \min\{t,l-t\}/c$ for every $t\in [0,l]. $
\end{defi}
\begin{thm}[Dyda, Lehrb\"ack and V\"ah\"akangas]\label{fhsm_jhn_dom}
Assume that $\Om \subsetneq \rN$ is an unbounded c-John Domain, $0<s<1$ and $1<p\leq q\leq \frac{Np}{N-ps}.$ Let $\beta \in \ro$ be such that
\begin{align*}
\dim_A(\partial\Om) < \min\left\{\frac{q}{p}\bct{N-sp+\beta},N-\frac{\beta}{p-1}\right\}.
\end{align*}
  Then for any $\tau \in (0,1)$ there exists a constant $C= C(\beta,\tau,N,s,p,c)>0$ such that for any $u \in \cup_{1\leq r <\infty}L^r(\Om) $ we have 
  \begin{align*}
 \bct{ \int_{\Om} \md{u(x)}^q \delta_{\partial\Om}^{\frac{q}{p}\bct{N-sp+\beta}-N}(x)dx}^{\frac{p}{q}} \leq C \int_{\Om}\int_{B^N_{\tau \delta_{\partial\Om}(x)}(x)}\frac{\md{u(x)-u(y)}^p}{\md{x-y}^{N+ps}} dy \delta_{\partial\Om}^\beta(x) dx,
  \end{align*}
where $\delta_{\partial\Om}^\beta(x) : = \operatorname{dist}(x,\partial\Om).$
\end{thm}
Choosing $\beta = 2s-k$, $p=2$, $q:=\frac{2(N-t)}{N-2s}$, $0\leq t<2s$, $0<s<1$, $\Om = \rk\setminus\{0\}\times \rNk$, $2\leq k\leq N-2$ in Theorem \ref{fhsm_jhn_dom} and using \eqref{grndst_rep} we get the inequality \eqref{FHSM_ineq}. Also notice that, choosing $\beta =0$ and $q=p=2$ we get the inequality \eqref{FHM_ineq} with a rough constant. For future reference, let us clarify that, as a consequence of the above discussion we get the following equivalent version of inequality \eqref{FHSM_ineq}.
\begin{align}\label{EFHSM_ineq}
C\bct{\int_{\rN}\frac{\md{u(x)}^{2^*_t}}{\md{x'}^{t+\ta2^*_t}}dx}^{\frac{2}{2^*_t}} \leq \int_{\rN}\int_{\rN}\frac{|u(x)-u(y)|^2}{|x-y|^{N+2s}|x'|^{\ta}|y'|^{\ta}} dx dy, \ \forall u\in C_c^{0,1}(\rN_k),
\end{align} 
where, $0<\ta \leq \frac{k-2s}{2}$, $0\leq t<2s$, $0<s<1$, $2\leq k \leq N-2$, $2^*_t = \frac{2(N-t)}{N-2s}$ and $C>0$ is constant depending only on $N,s,k,t$ and $\ta$.
\section{Existence of solution}\label{exis}
In this section, we will prove Theorem \ref{exis_FHSM} and Theorem \ref{exisb_FHSM}. 
\subsection{Proof of Theorem \ref{exis_FHSM}}
\begin{proof}
It is evident that $\mcal{\dot{H}}_\beta^s(\rN) = \dot{H}^s(\rN)$ for $\beta < \mcal{C}_{N,k,s}$. In fact, $\left[u\right]_{s,\beta,\rN}$ is an equivalent norm in $\dot{H}^s(\rN).$
For $u,v \in \dot{H}^s(\rN)$, let us define 

\begin{align*}
E[u,v]: = \int_{\rN}\int_{\rN} \frac{\bct{u(x)-u(y)}\bct{v(x)-v(y)}}{\md{x-y}^{N+2s}}dx dy -\beta \int_{\rN} \frac{uv}{\md{x'}^{2s}}dx.
\end{align*}
We will proof the theorem by dividing into two cases.\\
\textbf{Case I : t=0.} In this case we have 
\begin{align*}
S_k(\beta): = S^0_k(\beta) = \displaystyle \sup_{u\in \dot{H}(\rN)\setminus \{0\}} \frac{\int_{\rN}\md{u}^{2^*}}{\bct{E[u,u]}^{\frac{2^*}{2}}}.
\end{align*}
Let $\{u_m\}$ be any maximizing sequence. If necessary by normalizing we can assume $E[u_m,u_m] = 1, \ \forall m \in \mathbb{N}.$ In addition to this, we could also assume that $u_{m'}$s are nonnegative and radially symmetric in the first variable. This follows from inequality \eqref{fpol_sze} stated in Theorem \ref{mas_inq}. 

For now our aim is to show that, there exist $R_m>0$ and $x''_m\in \rNk$ such that the sequence $\{v_m\}$ defined by 
\begin{align*}
v_m(x',x''):= R_m^{\frac{N-2s}{2}}u_m(R_mx',R_mx''+x_m'')
\end{align*}
weakly converges to some non zero $v\in \dot{H}(\rN).$ 
Using the fact that $\nrm{u_m}_{L^{2^*}(\rN)}^{2^*} \rightarrow S_k(\beta)$ and \eqref{im_sob1} we have that there exists $R_m>0$ and $x_m\in \rN$ such that 
\begin{align*}
\frac{1}{R_m^{2s}} \int_{B^N_{R_m}(x_m)} \md{u_m(y)}^2 dy \geq C>0,
\end{align*} 
for some constant $C$ independent of $m$. Hence we have the following:

\begin{align}\label{com_est1}
\frac{1}{R_m^{2s}} \int_{B^{N-k}_{R_m}(x''_m)}\int_{B^k_{R_m}(x'_m)} \md{u_m(y',y'')}^2 dy' dy'' \geq C>0 ,\ \forall m \in \mathbb{N}.
\end{align}

Now let us define 
\begin{align*}
v_m(y',y'') := R_m^{\frac{N-2s}{2}}u_m(R_my',R_my''+x_m'').
\end{align*}
Then from \eqref{com_est1} we have 
\begin{align}\label{com_est2}
\int_{B_1^{N-k}(0)}\int_{B_1^k(\bar{x}_m')} \md{v_m(y)}^2 dy \geq C >0 , \ \forall m \in \mathbb{N},
\end{align}
where $\bar{x}'_m := \frac{x'_m}{R_m}.$ Clearly, $E[v_m,v_m] = E_[u_m,u_m] = 1.$ Hence by compactness we can assume the following: upto a subsequence 
\begin{itemize}
\item[(i)] $v_m \rightharpoonup v$ in $\dot{H}^s(\rN),$ 
\item[(ii)] $v_m \rightarrow v$ in $L^2_{loc}(\rN)$ and, 
\item[(iii)] $v_m \rightarrow v$ a.e.
\end{itemize}
We will prove that $v\neq 0$ by considering the following two cases.
\begin{itemize}
\item[Case I:] Upto a subsequence $\md{\bar{x}'_m} \rightarrow \infty$.  
\item[CaseII:] $\{\bar{x}_m'\}$ is bounded sequence in $\rk.$
\end{itemize}
\par In the first case, by using radial symmetry of $v_m$ in the first variable  we have 
\begin{align*}
\int_{B^{N-k}_1(0)} \int_{B^k_2(0)} \md{v_m(y',y'')}^2dy' dy'' &\geq \int_{B^{N-k}_1(0)} \int_{B^k_1(0)} \md{v_m(y'+\bar{x}_m',y'')}^2dy' dy'' \\
&= \int_{B^{N-k}_1(0)} \int_{B^k_1(\bar{x}_m')} \md{v_m(y',y'')}^2dy' dy''
\geq C>0. 
\end{align*}
Hence by passing to the limit we get 
\begin{align*}
\int_{B^{N-k}_1(0)} \int_{B^k_2(0)} \md{v(y',y'')}^2dy' dy'' \geq C>0,  
\end{align*}
which clearly shows that $v\neq 0.$  
\par In the second case, upto a subsequence $\bar{x}_m' \rightarrow x_0'.$ Then $B_1^k(\bar{x}_m') \subset \Omega$, for sufficiently large open bounded set $\Omega$ in $\rk$. So, using \eqref{com_est2} we get
\begin{align*}
\int_{B_1^{N-k}}\int_\Omega \md{v_m}^2 dy \geq \int_{B_1^{N-k}}\int_{B_1^k(\bar{x}_m')} \md{v_m}^2 dy \geq C > 0.
\end{align*}
Then again as before, passing to the limit we conclude that $v\neq 0.$ The rest of the proof is fairly standard. For the sake of the completeness, let us add the argument. By weak convergence of $v_m$ to $v$ we have 
\begin{align*}
E[v,v]+E[v_m-v,v_m-v] = 1 +o(1),
\end{align*}
as $m\rightarrow \infty.$ Also by scale invariance and Brezis-Lieb lemma \cite{BrLb} we have 
\begin{align*}
S_k(\beta) &= \lim_{m\rightarrow \infty} \int_{\rN} \md{v_m}^{2^*} dy \\
&= \lim_{m\rightarrow \infty}\left[\int_{\rN}\md{v}^{2^*}dy+\int_{\rN} \md{v_m-v}^{2^*}dy\right] \\ &\leq S_k(\beta) \bct{E[v,v]}^{\frac{2^*}{2}} + S_k(\beta) \bct{\lim_{m\rightarrow \infty} E[v_m-v,v_m-v]}^\frac{2^*}{2} \\ &\leq S_k(\beta)\bct{E[v,v]+\lim_{m\rightarrow \infty}E[v_m-v,v_m-v]}^{\frac{2^*}{2}}\leq S_k(\beta).
\end{align*}
Hence we must have equality everywhere in the above estimate. So, $S_{k}(\beta) = \frac{\int_{\rN}\md{v}^{2^*}dy}{\bct{E[v,v]}^\frac{2^*}{2}}.$ 
This proves the Case I, i.e. the existence of maximizer for $t=0.$ \\
\textbf{Case II: $0<t<2s.$} A careful inspection of the preceding proof yields that we only need the inequality \eqref{im_sob2}, which was proved in Proposition \ref{im_sob}, to conclude the proof. 
\end{proof}

\begin{rem}\label{betazero}
In the proof of Theorem \ref{exis_FHSM}, we have essentially used the fact that $\left[.\right]_{s,\rN}$ is an equivalent norm in $\mcal{\dot{H}}^s_\beta(\rN)$ along with the improved Sobolev inequality \eqref{im_sob2}. Hence we remark that, using the same arguments given there, we can conclude, 
$\overline{S}_k^t := \bct{1/S_k^t(0)}^{\frac{2}{2^*_t}}$ is attained by a nontrivial and nonnegative function $u_0 \in \dot{H}^s(\rN)$. Hence, $\kappa_k^t<\overline{S}^t_k.$
\end{rem}
\subsection{Existence of Solution For $\beta =\mcal{C}_{N,k,s}$}\label{crit_space}
\par Firstly, we observe that if $\phi \in C_c^\infty(\rN)$ with $\phi =1$ in $B^N_1(0)$, then $\phi \in \mcal{\dot{H}}^{s,\alpha}(\rN)$ ( See Lemma \ref{rep_sob_sp} of Appendix). Then, by ground state representation \eqref{grndst_rep}, we conclude that $\frac{\phi}{|x'|^\alpha} \in \mcal{\dot{H}}^s_\beta(\rN)$. Since, $\alpha = (k-2s)/2$, so we can see that $\frac{\phi}{|x'|^{\alpha}} \notin \dot{H}^s(\rN)$.  This shows that,  $\mcal{\dot{H}}^s_\beta(\rN)$ is much bigger than $\dot{H}^s(\rN)$. Because of this fact, we were not able to follow the same arguments as in the case of $\beta < \mcal{C}_{N,k,s}$. We have used Ekeland's Variation principle to prove the Theorem \ref{exisb_FHSM}. We start by the following compactness result.
\begin{lem}\label{compact_lem}
Let $\Om$ be a bounded open set in $\rN$. Then the following inclusion is compact.
\begin{align*}
\mcal{\dot{H}}^{s,\alpha}(\rN) \hookrightarrow  L^2\bct{\Om;\frac{1}{|x'|^{2\alpha}}},
\end{align*}
where $L^2\bct{\Om;\frac{1}{|x'|^{2\alpha}}}$ is the set of all measurable function $f:\rN \rightarrow \ro$ such that \\ $\frac{f}{|x'|^{\alpha}}\in L^2(\Om).$
\end{lem}

\begin{proof}
Let $u\in \mcal{\dot{H}}^{s,\alpha}(\rN).$ Then by H\"older inequality we have 
\begin{align}\label{compact_lem1} 
\int_{\Om}\frac{u^2}{|x'|^{2\alpha}}  &\leq |\Om|^{\frac{2s}{N}}  \bct{\int_{\Om}\frac{|u|^{2^*}}{|x'|^{\alpha2^*}}}^{\frac{2}{2^*}} \notag \\ &\leq \md{\Om}^{\frac{2s}{N}} \frac{1}{\kappa_k^0} \int_{\rN}\int_{\rN}\frac{\md{u(x)-u(y)}^2}{\md{x-y}^{N+2s}|x'|^\alpha|y'|^\alpha} dxdy,
\end{align}
where $\md{\Om}$ denotes the $N-$dimensional Lebesgue measure of $\Om$. Hence, the inclusion, $\mcal{\dot{H}}^{s,\alpha}(\rN) \hookrightarrow  L^2\bct{\Om;\frac{1}{|x'|^{2\alpha}}}$ is continuous. To prove the compactness, it is enough to show that $u_m \rightarrow 0$ in $L^2(\Om,\frac{1}{|x'|^{2\alpha}})$, whenever $u_m \in \mcal{\dot{H}}^{s,\alpha}(\rN)$ and $u_m$ converges to zero weakly in $\mcal{\dot{H}}^{s,\alpha}(\rN)$, as well as in $L^2(\Om,\frac{1}{|x'|^{2\alpha}})$. To prove this, we consider, $\phi_\epsilon \in C^\infty(\rk; [0,1])$ such that  
\begin{align*}
\phi_\epsilon = \begin{cases}0, \text{ if } |x'|<\epsilon\\ 1, \text{ if } |x'| >2\epsilon.
\end{cases}
\end{align*}
and define $\Om_\epsilon : = \Om\cap B^k_\epsilon(0)$. Without loss of generality, we could assume that $\Om\setminus \Om_\epsilon$ is a domain with Lipschitz boundary.  Also, let $M>0$ be such that $\left[\left[u_m\right]\right]^2_{s,\tilde{\alpha},\rN} < M$  and $\nrm{\frac{u_m}{|x'|^\alpha}}_{L^2(\Om)}<M$. Now, note that, 
\begin{align*}
\int_{\Om\setminus \Om_\epsilon}\int_{\Om \setminus \Om_\epsilon} \frac{\md{u_m(x)-u_m(y)}^2}{|x-y|^{N+2s}}dx dy &\leq C  \int_{\Om\setminus \Om_\epsilon}\int_{\Om \setminus \Om_\epsilon} \frac{\md{u_m(x)-u_m(y)}^2}{|x-y|^{N+2s}|x'|^\alpha|y'|^\alpha}dx dy \leq C \\ \text{and } \int_{\Om\setminus \Om_\epsilon} |u_m|^2 dx &\leq  C\int_{\Om} \frac{u_m^2}{|x'|^{2\alpha}} \leq C,
\end{align*}
where (and for the rest of the proof) $C>0$ is a constant depending on $N,s,k, \Om$ and $M$. So, $u_m$ is bounded sequence in $H^s(\Om\setminus \Om_\epsilon)$. Hence, by compactness, (see \cite{NPV}) $u_m \rightarrow u$ in $L^2(\Om\setminus\Om_\epsilon)$. By weak convergence of $u_m$ to $0$ in $L^2(\Om;\frac{1}{|x'|^{2\alpha}})$ we conclude that $u=0$ and  
\begin{align*}
\int_{\Om} \frac{\phi^2_\epsilon u^2_m}{|x'|^{2\alpha}} dx = \frac{1}{\epsilon^{2\alpha}} o(1), \text{ as } m \rightarrow \infty.
\end{align*}
Also, by \eqref{compact_lem1}
\begin{align*}
\int_{\Om} \md{(1-\phi_{\epsilon})u_m}^2\frac{dx}{|x'|^{2\alpha}} \leq C \md{\Om_\epsilon}^{\frac{2s}{N}}.
\end{align*}
Since,
\begin{align*}
\int_{\Om} u_m^2 \frac{dx}{|x'|^{2\alpha}} \leq C\left[\int_{\Om} \frac{\phi^2_\epsilon u^2_m}{|x'|^{2\alpha}} dx  +\int_{\Om} \md{(1-\phi_{\epsilon})u_m}^2\frac{dx}{|x'|^{2\alpha}}\right],
\end{align*}
so first letting $m\rightarrow \infty$ and then letting $\epsilon \rightarrow 0$ in the last inequality, we see that $u_m\rightarrow 0$ in $L^2(\Om;\frac{1}{|x'|^{2\alpha}})$.  This proves the lemma.

\end{proof}
Next, we need the following lemma.
\begin{lem}\label{lwr_ord_est}
Let $\{u_m\}$ be a bounded sequence in $\mcal{\dot{H}}^{s,\alpha}(\rN)$ such that $u_m \rightharpoonup 0$. Then for any $\psi \in C_c^\infty(\rN)$ with $\operatorname{supp}\psi \subset B_{R_0}^N(0)$, for some $R_0>0$, we have 
\begin{align}\label{lwr_ord_est1}
&\int_{\rN}\int_{\rN} \frac{\md{(\psi u_m)(x)-(\psi u_m)(y)}^2}{|x-y|^{N+2s}|x'|^\alpha|y'|^\alpha}dx dy \notag\\&\leq  \int_{\rN}\int_{\rN} \frac{\bct{(\psi^2 u_m)(x)-(\psi^2 u_m)(y)}\bct{u_m(x)-u_m(y)}}{|x-y|^{N+2s}|x'|^\alpha|y'|^\alpha}dx dy +\sigma_{R,m} +\tilde{\sigma}_R,
\end{align}
for any $R>R_0+2$ and $m\in \mathbb{N}$. Here, for each $R$, $\displaystyle{\lim_{m\rightarrow \infty}}\sigma_{R,m} =0$ and $\displaystyle{\lim_{R\rightarrow \infty}} \tilde{\sigma}_R =0$.
\end{lem}
\begin{proof}
Since, $R>R_0$, so, $\operatorname{supp} \psi \subset B^N_{R-2}(0)$. Now, note that,
\begin{align*}
&\bct{(\psi^2 u_m)(x)-(\psi^2 u_m)(y)}\bct{u_m(x)-u_m(y)}\\ &= \md{(\psi u_m)(x)-(\psi u_m)(y)}^2 - u_m(x)u_m(y) \md{\psi(x)-\psi(y)}^2.
\end{align*}
So, using $2ab\leq (a^2+b^2)$, with $a= \frac{|u_m(x)|}{|x'|^{\alpha}}$ and $b= \frac{|u_m(y)|}{|y'|^{\alpha}}$, we estimate
\begin{align}\label{lwr_ord_est2}
I_m &:= \md{\int_{\rN}\int_{\rN} \frac{u_m(x)u_m(y) \md{\psi(x)-\psi(y)}^2}{|x-y|^{N+2s}|x'|^\alpha|y'|^\alpha}dxdy}\notag\\  &\leq \int_{\rN}\frac{u_m^2(x)}{|x'|^{2\alpha}}\int_{\rN} \frac{\md{\psi(x)-\psi(y)}^2}{|x-y|^{N+2s}}dy dx \leq I_{m,1}+I_{m,2}
\end{align}
where 
\begin{align*}
I_{m,1} &:= \int_{|x|<R} \frac{u_m^2(x)}{|x'|^{2\alpha}}\int_{\rN} \frac{\md{\psi(x)-\psi(x-y)}^2}{|y|^{N+2s}} dy dx\text{ and } \\I_{m,2}&:= \int_{|x|>R}\frac{u_m^2(x)}{|x'|^{2\alpha}}\int_{\rN} \frac{\md{\psi(x)-\psi(x-y)}^2}{|y|^{N+2s}}dydx.
\end{align*}
Now, since 
\begin{align*}
\int_{\rN} \frac{\md{\psi(x)-\psi(x-y)}^2}{|y|^{N+2s}}dy \leq C\left[\nrm{\nabla \psi}_{L^\infty(\rN)} + \nrm{\psi}_{L^\infty(\rN)}\right],
\end{align*}
for some constant $C>0$ depending only on $N$, so using Lemma \eqref{compact_lem} we conclude that, for each $R>0$, $I_{m,1} = \sigma_{R,m}$, where $\displaystyle{\lim_{m\rightarrow \infty}}\sigma_{R,m} =0.$ To estimate $I_{m,2},$ we notice that, if $|x|>R$, $R>R_0+2$ and $|y|<1$ then $\psi(x)= \psi(x-y)=0$ Hence,
\begin{align}\label{lwr_ord_est3}
I_{m,2} &= \int_{|x|>R} \frac{u_m^2(x)}{|x'|^{2\alpha}}\int_{|y|>1} \frac{\md{\psi(x-y)}^2}{|y|^{N+2s}} dy dx \notag\\ &= \int_{|y|>1}\frac{1}{|y|^{N+2s}} \int_{|x|>R} \frac{u_m^2(x)}{|x'|^{2\alpha}} |\psi(x-y)|^2 dx dy
\end{align}
First using H\"oder inequality, then using \eqref{EFHSM_ineq} and boundedness of $u_m$ in $\mcal{\dot{H}}^{s,\alpha}(\rN)$, we get from \eqref{lwr_ord_est3}
\begin{align*}
I_{m,2} \leq C\int_{|y|>1} \frac{F_R(y)}{|y|^{N+2s}} dy,
\end{align*}
where, $F_R(y) : = \bct{\int_{|x|>R}\md{\psi(x-y)}^{N/s}dx}^{2s/N}$. Then using Dominated Convergence Theorem we have $I_{m,2} \leq \tilde{\sigma}_R$, for all $m\in \mathbb{N}$ and $\displaystyle{\lim_{R\rightarrow \infty}} \tilde{\sigma}_R =0.$ Hence, using \eqref{lwr_ord_est2} we conclude the lemma.
\end{proof}
Next, lemma shows that if the interior concentration happens, then it will happen away from the singular set.
\begin{prop}\label{con_lem}
Let $\{u_m\}$ be a nonnegative bounded sequence in $\mcal{\dot{H}}^{s,\alpha}(\rN)$ such that 

\begin{itemize}
\item[(i)] $\int_{\rN} \frac{\md{u_m}^{2^*_t}}{|x'|^{t+\alpha2^*_t}} dx = \bct{\kappa_k^t}^\frac{2^*_t}{2^*_t-2} $ \\ 
\item[(ii)] $L^s_\alpha u_m =\frac{\md{u_m}^{2^*_t-2}u_m}{|x'|^{t+\alpha2^*_t}} +f_m$, weakly in $\mcal{\dot{H}}^{s,\alpha}(\rN)$,
\end{itemize}
where $f_m\rightarrow 0$ in the dual of $\mcal{\dot{H}}^{s,\alpha}(\rN)$. If for any, $R_m>0$ and $\eta_m\in \rNk$, \\ $\bar{u}_m(x',x''): = R_m^{(k-N)/2}u_m(x'R_m^{-1},x''R_m^{-1}+\eta_m)$ converges to zero weakly in $\mcal{\dot{H}}^{s,\alpha}(\rN)$, then there exist $R_m>0$ and $\eta_m\in \rNk$ such that 
\begin{align*}
\liminf\limits_{m\rightarrow\infty} \int_{\Om} \frac{\md{\bar{u}_m}^{2^*_t}}{|x'|^{t+\alpha2^*_t}} dx>0, 
\end{align*}
 where $\Om: = \{(x',x'')\in\rN: \frac{1}{2}<|x'|<1, \ |x'|<1\}$.

\end{prop}

\begin{proof}
 We can choose $\eta_m \in \rNk$ and $R_m>0$ such that
\begin{align}\label{con_lem1}
Q_m(1):= \displaystyle{\sup_{\eta \in \rNk}} 
\int_{B_1^k(0)\times B_1^{N-k}(\eta)} \frac{\md{\bar{u}_m}^{2^*_t}}{|x'|^{t+2^*_t}}dx &= \int_{B_1^k(0)\times B_1^{N-k}(\eta)} \frac{\md{\bar{u}_m}^{2^*_t}}{|x'|^{t+2^*_t}}dx \notag\\&=\frac{ \bct{\kappa_k^t}^\frac{2^*_t}{2^*_t-2}}{2}.
\end{align}
 Then, clearly, $\{\bar{u}_m\}$ satisfies 
 \begin{align}\label{con_lem2}
 L^s_\alpha\bar{u}_m = \frac{\md{\bar{u}_m}^{2^*_t-2}\bar{u}_m}{|x'|^{t+\alpha2^*_t}} + \bar{f}_m, \text{ weakly in } \mcal{\dot{H}}^{s,\alpha}(\rN),
 \end{align}
where $\bar{f}_m\rightarrow 0$ in the dual of $\mcal{\dot{H}}^{s,\alpha}(\rN)$.
Now, there exist $N_0 \in \mathbb{N}$ and $\eta_1, \dots, \eta_{N_0}\in \rNk$ such that  $\overline{B_1^{N-k}(0)} \subset\displaystyle{\cup_{j=1}^{N_0}}B^{N-k}_{1/2}(\eta_j)$. Let $\psi_1,\dots,\psi_{N_0}\in C_c^\infty(\rN;[0,1])$ be such that 
\begin{align*}
\psi_j &= 1, \text{ on } B_{\frac{1}{2}}^k(0)\times B_{\frac{1}{2}}^{N-k}(\eta_j) \text{ and }  \operatorname{supp} \psi_j \subset B_{1}^k(0)\times B_1^{N-k}(\eta_j).
\end{align*}
 Now, using  $\psi_j = 1$, on  $B_{\frac{1}{2}}^k(0)\times B_{\frac{1}{2}}^{N-k}(\eta_j)$, translation invariance of norms in the second variable (i.e. $x''$) and employing the same calculations as in the Lemma \ref{cpt_apprx}, we can assert that $\psi_j^2\bar{u}_m$ is a bounded sequence in $\mcal{\dot{H}}^{s,\alpha}(\rN)$.  Hence using $\psi_j^2\bar{u}_m$, as a test function in \eqref{con_lem2} we arrive at
 \begin{align}\label{con_lem3}
 \int_{\rN}\int_{\rN} \frac{\bct{(\psi_j^2 \bar{u}_m)(x)-(\psi_j^2 \bar{u}_m)(y)}\bct{\bar{u}_m(x)-\bar{u}_m(y)}}{|x-y|^{N+2s}|x'|^\alpha|y'|^\alpha}dx dy = \int_{\rN} \frac{\md{\bar{u}_m}^{2^*_t-2}\bar{u}_m^2\psi_j^2}{|x'|^{t+\alpha2^*_t}} + \bar{\sigma}_m,
 \end{align}
where $\bar{\sigma}_m \rightarrow 0$ as $m \rightarrow \infty$. Since $\bar{u}_m \rightharpoonup 0$ in $\mcal{\dot{H}}^{s,\alpha}(\rN)$, so, using Lemma \ref{lwr_ord_est}, \eqref{con_lem3} and H\"older inequlaity we conclude that 
\begin{align}\label{con_lem4}
&\int_{\rN}\int_{\rN} \frac{\md{(\psi_j \bar{u}_m)(x)-(\psi_j \bar{u}_m)(y)}^2}{|x-y|^{N+2s}|x'|^\alpha|y'|^\alpha}dx dy \notag\\&\leq \bct{\int_{B_1^k(0)\times B_1^{N-k}(\eta_j)}\frac{\md{\bar{u}_m}^{2^*_t}}{|x'|^{t+\alpha2^*_t}}}^{\frac{2^*_t-2}{2^*_t}}\bct{\int_{\rN}\frac{\md{\psi_j\bar{u}_m}^{2^*_t}}{|x'|^{t+\alpha2^*_t}}}^{\frac{2}{2^*_t}} + \bar{\sigma}_m +\sigma_{R,m} +\tilde{\sigma}_R,
\end{align}
where for any large $R$ (depending only on $j$) $\displaystyle{\lim_{m\rightarrow \infty}}\sigma_{R,m} =0$, $\displaystyle{\lim_{R\rightarrow \infty}} \tilde{\sigma}_R =0$ and $\displaystyle{\lim_{m\rightarrow \infty}} \bar{\sigma}_m = 0$.  Hence, using \eqref{EFHSM_ineq}, \eqref{con_lem1}  and \eqref{con_lem4} we have
 \begin{align*}
\kappa_k^t \bct{\int_{\rN}\frac{\md{\psi_j\bar{u}_m}^{2^*_t}}{|x'|^{t+\alpha2^*_t}}}^\frac{2}{2^*_t} \leq  \frac{\kappa_k^t}{2^{\frac{2^*_t-2}{2^*_t}}}\bct{\int_{\rN}\frac{\md{\psi_j\bar{u}_m}^{2^*_t}}{|x'|^{t+\alpha2^*_t}}}^\frac{2}{2^*_t}+ \sigma_{R,m} +\tilde{\sigma}_R + \bar{\sigma}_m.
 \end{align*}
First letting, $m \rightarrow \infty$, then letting $R\rightarrow \infty$  we conclude that for any $j = 1,\dots,N_0$
\begin{align*}
\displaystyle{\lim_{m\rightarrow \infty}}\int_{\rN}\frac{\md{\psi_j\bar{u}_m}^{2^*_t}}{|x'|^{t+\alpha2^*_t}}dx = 0.
\end{align*}
 Hence 
 \begin{align*}
 \int_{B_{\frac{1}{2}}^k(0)\times B_1^{N-k}(0)}\frac{\md{\bar{u}_m}^{2^*_t}}{|x'|^{t+\alpha2^*_t}}dx \leq  \displaystyle{\sum_{j=1}^{N_0}} \int_{B_{\frac{1}{2}}^k(0)\times B_{\frac{1}{2}}^{N-k}(\eta_j)}\frac{\md{\bar{u}_m}^{2^*_t}}{|x'|^{t+\alpha2^*_t}}dx \rightarrow 0, \text{ as } m\rightarrow \infty.
 \end{align*}
Finally, using \eqref{con_lem1} we conclude the Proposition. 
\end{proof}
\subsection{Proof of Theorem \ref{exisb_FHSM}}
\begin{proof}
We can take a nonnegative minimizing sequence $\{u_m\}$ in $\mcal{\dot{H}}^{s,\alpha}(\rN)$ satisfying the following.
\begin{itemize}
\item[(i)] $\int_{\rN} \frac{\md{u_m}^{2^*_t}}{|x'|^{t+\alpha2^*_t}} dx = \bct{\kappa_k^t}^{\frac{2^*_t}{2^*_t-2}}$,
\item[(ii)] $\left[\left[u_m\right]\right]^2_{s,\alpha,\rN} = \bct{\kappa_k^t}^{\frac{2^*_t}{2^*_t-2}}$ + o(1), as $m\rightarrow \infty$,
\item[(iii)]  $L^s_\alpha u_m = \frac{\md{u_m}^{2^*_t-2}u_m}{|x'|^{t+\alpha2^*_t}} +f_m$  weakly in $\mcal{\dot{H}}^{s,\alpha}(\rN)$, 
\end{itemize}
where (iii) is a consequence of Ekeland's Variation Principle and $f_m \rightarrow 0$ in the dual of $\mcal{\dot{H}}^{s,\alpha}(\rN)$. We claim that, there exist $R_m>0$ and $\eta_m\in \rNk$, such that upto a subsequence, $\bar{u}_m \rightharpoonup u \neq 0$ in $\mcal{\dot{H}}^{s,\alpha}(\rN),$ where, $\bar{u}_m$ is defined as in Proposition \ref{con_lem}. Once the claim is established, we can argue similarly as in  the proof of Theorem \ref{exis_FHSM} to complete the Theorem.
If possible, let us assume, that the claim is false.  Then Proposition \ref{con_lem} guarantees the existence of $R_m>0$ and $\eta_m\in \rNk$ such that $\bar{u}_m \rightharpoonup 0$ in $\mcal{\dot{H}}^{s,\alpha}(\rN)$, 
\begin{align*}
&\displaystyle{\lim_{m \rightarrow \infty}} \int_{\Om} \frac{\md{\bar{u}_m}^{2^*_t}}{|x'|^{t+\alpha2^*_t}} dx >0 \text{ and }\\ L^s_\alpha \bar{u}_m &= \frac{\md{\bar{u}_m}^{2^*_t-2}\bar{u}_m}{|x'|^{t+\alpha2^*_t}} +\bar{f}_m \text{ weakly in }  \mcal{\dot{H}}^{s,\alpha}(\rN),
\end{align*}
where $\bar{f}_m \rightarrow 0$ in the dual of $\mcal{\dot{H}}^{s,\alpha}(\rN)$ and $\Om: = \{(x',x'')\in\rN: \frac{1}{2}<|x'|<1, \ |x'|<1\}$.  Let $\psi \in C_c^\infty (\rN_k,[0,1])$, such that $\psi = 1$ in $\Om$. Since, $\psi$ is supported away from $\{x'=0\}$ so, we can perform a similar calculation, given in Lemma \ref{cpt_apprx} to conclude $\{\psi^2 \bar{u}_m\}$ is a bounded sequence in $\mcal{\dot{H}}^{s,\alpha}(\rN)$. Then, proceeding similarly as in the  proof of Proposition \ref{con_lem}, we can derive 
\begin{align}\label{exisb_FHSM1}
&\int_{\rN}\int_{\rN} \frac{\md{(\psi \bar{u}_m)(x)-(\psi \bar{u}_m)(y)}^2}{|x-y|^{N+2s}|x'|^\alpha|y'|^\alpha}dx dy \notag\\&\leq \bct{\int_{\rN}\frac{\md{\bar{u}_m}^{2^*_t}}{|x'|^{t+\alpha2^*_t}}}^{\frac{2^*_t-2}{2^*_t}}\bct{\int_{\rN}\frac{\md{\psi\bar{u}_m}^{2^*_t}}{|x'|^{t+\alpha2^*_t}}}^{\frac{2}{2^*_t}} + \bar{\sigma}_m +\sigma_{R,m} +\tilde{\sigma}_R,\notag \\ &\leq \kappa_k^t\bct{\int_{\rN}\frac{\md{\psi\bar{u}_m}^{2^*_t}}{|x'|^{t+\alpha2^*_t}}}^{\frac{2}{2^*_t}} +  \bar{\sigma}_m +\sigma_{R,m} +\tilde{\sigma}_R,
\end{align}
where for any large $R$ (depending only on the support of $\psi$) $\displaystyle{\lim_{m\rightarrow \infty}}\sigma_{R,m} =0$, $\displaystyle{\lim_{R\rightarrow \infty}} \tilde{\sigma}_R =0$ and $\displaystyle{\lim_{m\rightarrow \infty}} \bar{\sigma}_m = 0.$ Now, let us define $v_m := \frac{\psi\bar{u}_m}{|x'|^\alpha}$. Then by \eqref{grndst_rep}  and \eqref{exisb_FHSM1} we have 
\begin{align*}
\left[v_m\right]^2_{s,\rN} - \mcal{C}_{N,k,s} \int_{\rN}\frac{v^2_m}{|x'|^{2s}}  \leq \kappa_k^t\bct{\int_{\rN}\frac{\md{\psi\bar{u}_m}^{2^*_t}}{|x'|^{t+\alpha2^*_t}}}^{\frac{2}{2^*_t}} +  \bar{\sigma}_m +\sigma_{R,m} +\tilde{\sigma}_R.
\end{align*}
Since, $\psi$ is supported in $\rN_k$, using Lemma \ref{compact_lem} and inequality \eqref{FHSM_ineq} for $\beta = 0$, we have 
\begin{align}\label{exisb_FHSM2}
\overline{S}_k^t \bct{\int_{\rN}\frac{\md{\psi\bar{u}_m}^{2^*_t}}{|x'|^{t+\alpha2^*_t}}}^{\frac{2}{2^*_t}} \leq \kappa_k^t\bct{\int_{\rN}\frac{\md{\psi\bar{u}_m}^{2^*_t}}{|x'|^{t+\alpha2^*_t}}}^{\frac{2}{2^*_t}} +  \bar{\sigma}_m +\sigma_{R,m} +\tilde{\sigma}_R.
\end{align}
As pointed out in Remark \ref{betazero}, $\kappa_k^t< \overline{S}^t_k$. Hence, from \eqref{exisb_FHSM2} we conclude that 
\begin{align*}
\displaystyle{\lim_{m\rightarrow \infty}} \int_{\rN}\frac{\md{\psi\bar{u}_m}^{2^*_t}}{|x'|^{t+\alpha2^*_t}} = 0,
\end{align*}
which contradicts Proposition \ref{con_lem}. This proves the Theorem.
\end{proof}

\section{Qualitative Properties of Solution}\label{reg}
Main goal of this section is to prove Theorem \ref{reg_flap_cyl}. Following the ideas of \cite{AbMePe}, using \eqref{grndst_rep}, we could hide the singular weight in the operator. Therefore, working with the newly defined operator and using Moser iteration technique, we can prove the asymptotic estimate \eqref{asy_est}, given in Theorem \ref{reg_flap_cyl}. On the other hand, to prove the regularity of solution, we have used extension technique introduced by Cafarelli and Silvestre in \cite{CS}.

The proof of inequality \eqref{asy_est} is based on  Moser iteration technique

\subsection{$L^\infty$ Estimates of Solutions} 
In this subsection, we will prove some $L^\infty$ estimates of solution of \eqref{flap_cyl}, which seems  inevitable to give an unified proof of the cylindrical symmetry of solutions. We start by introducing the operator $L_{\tilde{\alpha}}^s$, which is defined on $\mcal{\dot{H}}^{s,\tilde{\alpha}}(\rN)$ by the following inner product :
\begin{align}\label{sing_flap}
\left<L^s_{\tilde{\alpha}}v, \phi\right> : = \int_{\rN}\int_{\rN}\frac{\bct{v(x)-v(y)}\bct{\phi(x)-\phi(y)}}{\md{x-y}^{N+2s}|x'|^{\tilde{\alpha}}|y'|^{\tilde{\alpha}}}dx dy,
\end{align}
for all $\phi \in \mcal{\dot{H}}^{s,\tilde{\alpha}}(\rN).$ Notice that as consequence of Theorem \ref{grndst_rep}, if $u$ satisfies \eqref{flap_cyl} then $U:= \mcal{P}_{\tilde{\alpha}}[u] := |x'|^{\tilde{\alpha}}u \in \mcal{\dot{H}}^{s,\tilde{\alpha}}(\rN)$ satisfies  
\begin{align}\label{sflap_cyl}
L^s_{\tilde{\alpha}} U  = \frac{U^{2^*_t-1}}{\md{x'}^{t+2^*_t\tilde{\alpha}}}
\end{align}
weakly, where $0<\tilde{\alpha}\leq \alpha.$ First, we recall the following lemma, from \cite{Kmann}, for convex functions, which will be used to derive a Kato type inequality for the newly defined operator
 $L^s_{\tilde{\alpha}}$, in the proof of Proposition \ref{linfty_est}.
\begin{lem}\label{kato_ineq}
Let $I\subset \mathbb{R}$ be an interval, $a,b \in I$, $\theta_1,\theta_2 \geq 0.$ If $f\in C^1(I)$ is convex, then 
\begin{align*}
(b-a)\bct{\theta_1f'(b)-\theta_2f'(a)} \geq \bct{f(b)-f(a)}\bct{\theta_1-\theta_1}.
\end{align*}
In particular, the following inequality holds:
\begin{align}\label{kato_ineq1}
(b-a)\bct{f(b)f'(b) -f(a)f'(a)} \geq \bct{f(b)-f(a)}^2.
\end{align}

\end{lem}
 
\begin{prop}\label{linfty_est}
Let $V\in \mcal{\dot{H}}^{s,\tilde{\alpha}}(\rN)$ be any non negative solution of \eqref{sflap_cyl}, where \\ $2\leq k\leq N-2 $, $0<s<1$, $0\leq t<2s$, $0<\tilde{\alpha}\leq \alpha$ and $2^*_t = \frac{2(N-t)}{N-2s}$. Then $V\in L^\infty(\rN).$  
\end{prop}
\begin{proof}
For $q\geq 1$ and $R>0$ we define
\begin{align*}
\phi(r):= \phi_{q,R}(r) :=
\begin{cases}
r^q , \ \text{if } 0\leq r \leq R \\ qR^{q-1}(r-R)+R^q, \ \text{if }r>R.
\end{cases}
\end{align*}
Clearly, $\phi_{q,R}$ is Lipschitz and $\phi_{q,R}(0) =0$ so, by Lemma \ref{rep_sob_sp}, $\phi_{q,R}(V)\in \mcal{\dot{H}}^{s,\tilde{\alpha}}(\rN)$. So, using \eqref{EFHSM_ineq} and Theorem \ref{grndst_rep} we have 
\begin{align}\label{mi1}
\int_{\rN}\int_{\rN} \frac{\md{\phi\bct{V(x)}-\phi \bct{V(y)}}^2}{\md{x-y}^{N+2s}} \frac{dxdy}{|x'|^{\tilde{\alpha}}|y'|^{\tilde{\alpha}}}\geq C_{N,s,k} \left[\int_{\rN}\md{\phi(V)}^{2_t^*} \frac{dx}{|x'|^{t+\tilde{\alpha}2_t^*}}\right]^{2/2_t^*}.
\end{align}
Notice that, $\phi\in C^1\bct{[0,\infty)}$ so,  using inequality \ref{kato_ineq1} we have 

\begin{align}\label{mi1.1}
&\int_{\rN}\int_{\rN} \frac{\md{\phi\bct{V(x)}-\phi \bct{V(y)}}^2}{\md{x-y}^{N+2s}} \frac{dxdy}{|x'|^{\tilde{\alpha}}|y'|^{\tilde{\alpha}}} \notag \\&\leq \int_{\rN}\int_{\rN} \frac{\bct{\phi\bct{V(x)}\phi'\bct{V(x)}-\phi \bct{V(y)}\phi'\bct{V(y)}}\bct{V(x)-V(y)}}{\md{x-y}^{N+2s}} \frac{dxdy}{|x'|^{\tilde{\alpha}}|y'|^{\tilde{\alpha}}}.
\end{align}
Now, observe that the function $g$ defined by $g := \phi\phi'$ is Lipschitz with $g(0) = 0.$ Hence, $g(V) = \phi(V)\phi'(V) \in \mcal{\dot{H}}^{s,\tilde{\alpha}}(\rN)$. So, using $g(V)$ as test function in \eqref{sflap_cyl} and then employing \eqref{mi1.1} we have

\begin{align}\label{mi2}
\int_{\rN}\int_{\rN} \frac{\md{\phi\bct{V(x)}-\phi \bct{V(y)}}^2}{\md{x-y}^{N+2s}} \frac{dxdy}{|x'|^{\tilde{\alpha}}|y'|^{\tilde{\alpha}}} &\leq \int_{\rN} \phi(V)\phi'(V)\frac{V^{2_t^*-1}}{|x'|^{t+\tilde{\alpha}2_t^*}}dx \notag \\ &\leq 2q\int_{\rN}\phi^2(V)V^{2_t^*-2}\frac{dx}{|x'|^{t+\tilde{\alpha}2_t^*}},
\end{align}
 where in the last inequality we have used $r\phi'(r)\leq 2q\phi(r)$. Combining \eqref{mi1} and \eqref{mi2} we have 
 \begin{align}\label{mi3}
  \left[\int_{\rN}\md{\phi(V)}^{2_t^*} \frac{dx}{|x'|^{t+\tilde{\alpha}2^*}}\right]^{2/2_t^*} \leq 2qC_{N,s,k} \int_{\rN}\phi^2(V)V^{2_t^*-2}\frac{dx}{|x'|^{t+\tilde{\alpha}2_t^*}}
 \end{align}
We estimate the R.H.S. of \eqref{mi3} in the following manner. Consider $q= 2_t^*/2$ and $a>0$ (to be chosen later). Then we have
\begin{align}\label{mi4}
2qC_{N,s,k} \int_{\rN}\phi^2(V)V^{2_t^*-2}\frac{dx}{|x'|^{t+\tilde{\alpha}2_t^*}}&= 2qC_{N,s,k}\int_{\{V\leq a\}} \phi^2(V)V^{2_t^*-2}\frac{dx}{|x'|^{t+\tilde{\alpha}2_t^*}} \notag \\ &+ 2qC_{N,s,k} \int_{\{V>a\}} \phi^2(V)V^{2_t^*-2}\frac{dx}{|x'|^{t+\tilde{\alpha}2_t^*}}.
\end{align}
Now, first using H{\"o}lder inequality and then choosing $a>0$ large enough we have 
\begin{align*}
2qC_{N,s,k} \int_{\{V>a\}} \phi^2(V)V^{2_t^*-2}\frac{dx}{|x'|^{t+\tilde{\alpha}2_t^*}} \leq \frac{1}{2} \left[\int_{\rN}\md{\phi(V)}^{2_t^*}\frac{dx}{|x'|^{t+\tilde{\alpha}2_t^*}}\right]^{\frac{2}{2_t^*}}
\end{align*}
Hence incorporating this estimate in \eqref{mi4} and then using \eqref{mi3} we have 
\begin{align*}
\left[\int_{\rN}\md{\phi(V)}^{2_t^*}\frac{dx}{|x'|^{t+\tilde{\alpha}2_t^*}}\right]^{\frac{2}{2_t^*}} &\leq qa^{2_t^*-2}C_{N,s,k} \int_{\rN} \frac{\phi^2(V)dx}{|x'|^{t+\tilde{\alpha}2_t^*}} \\ &\leq qa^{2_t^*-2}C_{N,s,k} \int_{\rN} \frac{|V|^{2_t^*}}{|x'|^{t+\tilde{\alpha}2_t^*}},
\end{align*}
where in the last inequality we have used $\phi(r)\leq r^q$ and $q=2^*/2.$ Now letting $R\rightarrow \infty$ and using Fatou's lemma we conclude 
\begin{align}\label{mi5}
\int_{\rN} \md{V}^{2_t^*\frac{2_t^*}{2}} \frac{dx}{|x'|^{t+\tilde{\alpha}2_t^*}} <\infty.
\end{align}
For $m\geq 1$ we define $\{q_m\}$ by
\begin{align}\label{mi6}
2q_{m+1} + 2_t^*-2 = 2_t^*q_m, \ q_1 = \frac{2_t^*}{2}.
\end{align}
Then using \eqref{mi3} and \eqref{mi6} we arrive at
\begin{align*}
\left[\int_{\rN}\md{\phi_{q_{m+1},R}(V)}^{2_t^*}\frac{dx}{|x'|^{t+\tilde{\alpha}2_t^*}}\right]^{\frac{2_t^*}{2}} \leq q_{m+1}C_{N,s,k} \int_{\rN} \frac{V^{q
_{m}2_t^*}}{|x'|^{t+\tilde{\alpha}2_t^*}} dx
\end{align*}
Again letting $R\rightarrow 0$ and using Fatou's lemma we conclude
\begin{align}\label{mi7}
\left[\int_{\rN} \frac{V^{q
_{m+1}2_t^*}}{|x'|^{t+\tilde{\alpha}2_t^*}} dx\right]^{\frac{1}{2_t^*(q_{m+1}-1)}} \leq \bct{q_{m+1}C_{N,s,k}}^{\frac{1}{2(q_{m+1}-1)}}\left[\int_{\rN} \frac{V^{q
_{m}2_t^*}}{|x'|^{t+\tilde{\alpha}2_t^*}} dx\right]^{\frac{1}{2_t^*(q_m-1)}}.
\end{align}
For $m\geq 1$, set 
\begin{align*}
I_m : = \left[\int_{\rN} \frac{V^{q
_{m}2_t^*}}{|x'|^{t+\tilde{\alpha}2_t^*}} dx\right]^{\frac{1}{2_t^*(q_m-1)}}, \ D_m = \bct{q_{m+1}C_{N,s,k}}^{\frac{1}{2(q_{m+1}-1)}}.
\end{align*}
 Then \eqref{mi7} gives $I_{m+1}\leq D_mI_m, \forall m \geq 1$. Taking $\log$ in both side and then iterating we get
 \begin{align}\label{mi8}
 \log I_{m+1} \leq \sum_{j=1}^m \log D_j + \log I_1.
 \end{align}
Since $q_1 >1$, it is easy to see $\sum_{j=1}^\infty \log D_j < C_{N,s,k}$. Hence using \eqref{mi5}, from \eqref{mi8} we get 
\begin{align*}
I_{m+1} \leq C_{N,s,k}, \ \forall m\geq1.
\end{align*}
 So, for any $R>0$ we have 
 \begin{align*}
 \left[\int_{|x|\leq R}V^{q_m2^*_t}dx\right]^\frac{1}{2_t^*(q_m-1)}\leq R^{\frac{t+\tilde{\alpha}2^*_t}{2^*_t(q_m-1)}}\left[\int_{|x|\leq R} \frac{V^{q
_{m}2_t^*}}{|x'|^{t+\tilde{\alpha}2_t^*}} dx\right]^{\frac{1}{2_t^*(q_m-1)}} \leq C_{N,s,k} R^{\frac{t+\tilde{\alpha}2^*_t}{2^*_t(q_m-1)}}.
 \end{align*}
Since $q_m \rightarrow \infty$ as $m \rightarrow \infty$, so $V \in L^\infty(B^N_R(0))$ and 
\begin{align*}
\nrm{V}_{L^\infty(B^N_R(0))} \leq C_{N,s,k}, \ \forall R>0.
\end{align*}
This proves the proposition.
\end{proof}
\begin{cor}\label{linfty_flap}
As a consequence of the preceding Proposition and Theorem \ref{grndst_rep}, we observe that if $u$ solves \eqref{flap_cyl} then $\mcal{P}_{\tilde{\alpha}}[u] \in L^\infty(\rN)$. Moreover, by Lemma 2.2 of \cite{FallWeth},  $\eqref{flap_cyl}$ is invariant under the Kelvin transformation i.e. $Ku(x):= \frac{1}{|x|^{N-2s}}u(x/|x|^2)$. Hence we have the following asymptotic estimate for any $u$ solving \eqref{flap_cyl}
\begin{align*}
u(x) \leq \frac{C}{|x'|^{\tilde{\alpha}}\bct{1+\md{x}^{N-2s-2\tilde{\alpha}}}} ,  \ \forall x\in \rN_k,
\end{align*}
where $C>0$ is constant, depends on $u$ but independent of $x\in \rN_k.$
\end{cor}

\subsection{Extension and Regularity}
Throughout this section, we will consider, $\beta < \mcal{C}_{k,s}.$ Let $u \in \dot{H}^s(\rN)$ be a positive solution of 
\begin{align*}
\bct{-\Delta}^s u - \beta \frac{u}{\md{x'}^{2s}} = \frac{u^{2_t^*-1}}{|x'|^t} \ \text{in } \rN.
\end{align*}
We consider, the $s$-Harmonic extension $U$ of $u$ defined by 
\begin{align}\label{s-har_ext}
U(x,r) : = d_{N,s}\int_{\rN}\frac{r^{2s}}{\bct{|x-\xi|^2+r^2}^{\frac{N+2s}{2}}}u(\xi) d\xi, \ \text{for } x\in \rN, \ r\in (0,\infty).
\end{align}
where the constant $d_{N,s}$ is chosen so that $\int_{\rN}\frac{dx}{\bct{1+|x|^2}^\frac{N+2s}{2}} = \frac{2^{1-2s}\Gamma(1-s)}{d_{N,s}\Gamma(s)}$. Then, (See \cite{CS}) for any bounded domain $\Om \subset \ro^{N+1}_+$, $U \in H^1(\Om,1-2s)$ and satisfies (weakly) 
\begin{align}\label{ext_flap_cyl}
\begin{cases}
\operatorname{div}\bct{r^{1-2s}\nabla U}= 0\ &\text{in } \Om,\\
-\lim_{r\rightarrow 0+}r^{1-2s}\partial_rU(x,r) = a(x)U(x,0) +b(x)\ &\text{on }\partial'\Om,
\end{cases}
\end{align}
where $\partial \Om'$ is the interior of $\partial \Om \cap \rN$, $b(x)=0$ and $a(x):= \bct{\frac{\beta}{|x'|^{2s}} + \frac{U^\frac{4s-2t}{N-2s}(x,0)}{|x'|^t}}.$ For $R>0$, we denote $Q_R := B_R^N(0)\times (0,R).$ We will use the following results from \cite{JLX} to prove the smoothness of solution away from the set $\{x'=0\}$.  

\begin{prop}\label{imp_int}
Let $a\in L^{N/2s}(B_1^N)$, $b\in L^p(B_1^N)$ with $p>2s$. Also, assume that  $0\leq \overline{U}\in H^1(Q_1,1-2s)$ is weak solution of \eqref{ext_flap_cyl} in $Q_1$. Then there exists $\delta >0$ which depends only on $n$ and $s$ such that if $\nrm{a^+}_{L^{n/2s}(B_1^N)}< \delta,$ then 
\begin{align*}
\nrm{\overline{U}(.,0)}_{L^q(\partial'Q_{1/2})} \leq C \bct{\nrm{\overline{U}}_{L^2(Q_1,1-2s)}+ \nrm{\nabla \overline{U}}_{L^2(Q_1,1-2s)}+\nrm{b^+}_{L^p(B_1^N)}},
\end{align*}
where $C>0$ depends only on $N,p,s,\delta$ and $q := \min\{\frac{2(N+1)}{N-2s}, \frac{N(p-1)}{(N-2s)p}, \frac{2N}{N-2s}\}.$
\end{prop}

\begin{prop}\label{har_inq}
Let $\overline{U} \in H^1(Q_1,1-2s)$ be a nonnegative weak solution of \eqref{ext_flap_cyl} and $a,b \in L^p(B_1)$ for some $p>N/{2s}.$ Then we have the following Harnack inequality:
\begin{align*}
\sup_{Q_{1/2}}{\overline{U}} \leq C \bct{\inf_{Q_{1/2}}\overline{U}+ \nrm{b}_{L^p(B_1)}},
\end{align*}
where $C>0$ depends only on $N,p,s$ and $\nrm{a}_{L^p(B_1)}.$ Consequently, there exists $\alpha \in (0,1)$ depending only on $N,p,s$ and $\nrm{a}_{L^p(B_1)}$ such that any weak solution $\overline{U}$ of \eqref{ext_flap_cyl} is in $C^{\alpha}\bct{\overline{{Q}_{1/2}}}.$
\end{prop}

\subsection{Proof of Theorem \ref{reg_flap_cyl}} 
\begin{proof}
Clearly, the first part of the theorem follows form Corollary \ref{linfty_flap}. To prove the second part, we take $x_0 \in \rN_k$ and $R_0>0$ be such that $\overline{B^N_{R_0}(x_0)} \subset \rN_k.$ Consider the $s$-harmonic extension $U$ of $u$ defined in \eqref{s-har_ext}. Clearly, $U$ is nonnegative and satisfies \eqref{ext_flap_cyl} with $b=0$, weakly in $B^N_{R_0}(x_0)\times (0,R_0).$ We define $V(x,r):= U(x_0+R_0x,R_0r)$ for $x\in B^N_1(0)$ and $0<r<1.$ Then, $V$ satisfies \eqref{ext_flap_cyl} with $a(x): = R_0^{2s}\bct{\frac{\beta}{\md{x_0'+R_0x'}^{2s}} + \frac{u^{2_t^*-2}(x)}{\md{x_0'+R_0x'}^t}}$ and $b(x)= 0.$ Since $\overline{B^N_{R_0}(x_0)} \subset \rN_k$, $a\in L^{\frac{N}{2s-t}}(B_1^N)$. Using Proposition \ref{imp_int} whenever needed, we conclude $a\in L^p(B^N_{1/2})$ with $p>N/{2s}$, for any $0\leq t<2s$. Hence by Proposition \ref{har_inq}, $V\in C^\alpha(\overline{{Q}_{1/4}}).$ So, by local Schauder estimates (See Theorem 2.11 of \cite{JLX}) and bootstrapping argument we have $V \in C^\infty(\overline{Q_{1/4}})$ which in turn implies $U \in C^\infty \left(\overline{B^N_{R_0}(x_0)}\times [0,R] \right).$ This implies $u\in C^\infty(\rN_k).$ This proves the theorem.
\end{proof}

\section{Cylindrical Symmetry of Positive solution}
\label{symm}
Our main goal of this section is to prove Theorem \ref{cyl_sym}.

First, we will prove the following strong maximum principle for antisymmetric function. Notice that, here we are not assuming any lower semi continuity of $u$ upto the boundary. To compensate this, we are assuming a global non negativity of $u$ on the half plane. The proof is a suitable adaptation of the techniques introduced by Silvestre in \cite{S}. For reader's convenience, we will add the proof.

\begin{lem}\label{mp_as}
Let $u \in L_s(\rN)$ and $(-\Delta)^su\geq 0$ in $\Om$, in the distributional sense , where $\Om \subset \Om_\lambda := \{x\in\rN: x_1<\lambda\}$ is open and bounded. Also, assume that $u\geq 0$ a.e. in $\Om_\lambda$ and $u$ is antisymmetric i.e. $u(x_\lambda) = -u(x)$ for a.e. $x\in \Om_\lambda$, where $x_\lambda := (2\lambda-x_1,x_2,\dots, x_N).$ Then either $u>0$ in $\Om$ or $u\equiv 0$ a.e. in $\rN.$ Moreover, the above result still remains true if we replace $x_1$ by $x_i$ for any $i\in \{1,2,\dots,N\}$ in the definition  of $\Om_\lambda.$ 
\end{lem}
\begin{proof}
 Since $(-\Delta)^su \geq 0$ in $\Om$ and $u \in L_s(\rN)$ so by Proposition 2.15 of \cite{S} we have, $u$ is lower semicontinuous in $\Om$ and satisfies the following
\begin{align}\label{mp_as1}
u(x_0) \geq \int_{\rN} u(x)\tau_\gamma(x-x_0) dx,
\end{align}
 for any $x_0 \in \Om$ and $\gamma < \operatorname{dist}(x_0,\partial\Om).$ Here, $\tau_\gamma(x) := (-\Delta)^s\Gamma_\gamma(x)$, $\Gamma_\gamma(x) : = \frac{\Gamma(x/\gamma)}{\gamma^{N-2s}}$, and $\Gamma$ is a $C^{1,1}$ regularization of $\Phi(x) : = \frac{1}{|x|^{N-2s}}$ such that 
 \begin{align*}
 \Gamma &\equiv \Phi \ \text{in } \rN\setminus B_1^N(0),\\
 \Gamma &\leq \Phi \ \text{in }  B_1^N(0). 
 \end{align*}
Next, we claim that $\tau_\gamma(x-x_0) \geq \tau_\gamma(x_\lambda-x_0)$, $\forall x\in \Om_\lambda.$ To prove this, we notice that, for $x\in \Om_\lambda \setminus B^N_\gamma(x_0)$
\begin{align*}
\tau_\gamma(x-x_0) &= P.V. \int_{\rN}\frac{\Gamma_\gamma(x-x_0)- \Gamma_{\gamma}(y)}{\md{x-x_0-y}^{N+2s}} dy\\
&= \int_{B_\gamma^N(x_0)} \frac{\Phi(y-x_0)-\Gamma_\gamma(y-x_0)}{\md{x-y}^{N+2s}} dy.
\end{align*}
To get the last equality we have used the fact that $\Phi$ is the fundamental solution of the fractional $s$-laplacian. Similarly, we have 
\begin{align*}
\tau_\gamma(x_\lambda-x_0)=\int_{B_\gamma^N(x_0)} \frac{\Phi(y-x_0)-\Gamma_\gamma(y-x_0)}{\md{x_\lambda-y}^{N+2s}} dy.
\end{align*}
Clearly, $\md{x-y} \leq \md{x_\lambda - y }$ if $x\in \Om_\lambda \setminus B_\gamma^N(x_0 )$ and $y \in B_\gamma^N(x_0 ).$ Hence, $\tau_\gamma(x-x_0) \geq \tau_\gamma(x_\lambda-x_0)$, $\forall x\in \Om_\lambda\setminus B_\gamma^N(x_0 ) .$ Now, let $x\in B_\gamma^N(x_0).$ Then, for $\gamma$ small 
\begin{align*}
\tau_\gamma(\tilde{x}-x_0)\leq \frac{\gamma^{2s}}{\md{\tilde{x}-x_0}^{N+2s}} \ ,\ \text{for } \md{\tilde{x}-x_0}\geq \frac{\operatorname{dist}(x_0,\partial \Om_\lambda)}{2},
\end{align*}
where $C>0$ is a generic constant. Hence, for $\gamma$ sufficiently small $\tau_\gamma(x_\lambda-x_0) \leq C$, whereas $\tau_\gamma(x-x_0) = \frac{1}{\gamma^N}\tau_1\bct{\frac{x-x_0}{\gamma}}\geq C.$ This concludes our claim. Now, consider 
\begin{align*}
\int_{\rN} u(x) \tau_\gamma(x-x_0) dx = \int_{\Om_\lambda}u(x)\tau_\gamma(x-x_0)dx +\int_{\rN\setminus \Om_\lambda} u(x)\tau_\gamma(x-x_0)dx.
\end{align*}
Since, for sufficiently small $\gamma$, $\tau_\gamma(x-x_0)\geq \tau_\gamma(x_\lambda-x_0)$, $\forall x\in \Om_\lambda,$ we have 
\begin{align*}
\int_{\Om_\lambda} u(x)\tau_\gamma(x-x_0)dx\geq \int_{\Om_\lambda}u(x)\tau_{\gamma}(x_\lambda-x_0)dx &= -\int_{\Om_\lambda} u(x_\lambda)\tau_{\gamma}(x_\lambda-x_0)dx\\ &= -\int_{\rN \setminus \Om_\lambda}u(y)\tau_{\gamma}(y-x_0)dy.
\end{align*}
Hence, $\int_{\rN}u(x)\tau_\gamma(x-x_0)dx\geq 0.$ Now, if possible, let us assume that there exists $x_0 \in \Om$ such that $u(x_0) \leq 0.$ Then, for $\gamma < \operatorname{dist}(x_0,\partial\Om) $ small enough we have from \eqref{mp_as1} 
\begin{align*}
0\geq u(x_0) \geq \int_{\rN} u(x)\tau_\gamma(x-x_0)dx \geq 0.
\end{align*}
So, $\int_{\rN} u(x)\tau_\gamma(x-x_0)dx= 0.$ Form here, using $u$ is antisymmetric and non negative on $\Om_\lambda$ one can easily conclude $u\equiv 0$ a.e. in $\Om_\lambda.$ This proves the lemma.
\end{proof}
\subsection{Proof of Theorem \ref{cyl_sym}.}
\begin{proof}
For $u,v \in \dot{\mcal{H}}_\beta^s(\rN)$ we define 
\begin{align*}
E[u,v] : = \int_{\rN}\int_{\rN} \frac{\bct{u(x)-u(y)}\bct{v(x)-v(y)}}{\md{x-y}^{N+2s}}dxdy - \beta \int_{\rN} \frac{u(x)v(x)}{|x'|^{2s}}dx.
\end{align*}
Also, for $\lambda\in \ro$ we define the following sets in $\rN$
\begin{align*}
\Om_{\lambda} := \{x_1<\lambda\}, \ \text{and } \Om'_{\lambda}: = \{x_{k+1}<\lambda\}. 
\end{align*}
Suppose $u$ solves \eqref{flap_cyl}. We will show that, $u$ is symmetric with respect to $\partial\Om_{0}= \{x_1=0\}$ and there exist $\lambda_0\in \ro$ such that for any fixed $x'\neq 0$, $u$ is symmetric with respect to $\partial\Om'_{\lambda_0}$. Once we show this, the rest of the proof will follow from standard arguments.
\par \textbf{Step 1 :} In this step we will show that $u$ is symmetric with respect to $\partial\Om_{0}$. Let $\lambda<0$ and $w_{\lambda} :=u-u_\lambda  $, where $u_{\lambda}(x):= u(x_\lambda)$ and $x_\lambda : = (2\lambda-x_1,x_2,\dots,x_N).$ We also define
\begin{align*}
v_\lambda(x):=
\begin{cases}
 \bct{u-u_\lambda}^+(x) ,\ \text{if } x\in \Om_\lambda \\  \bct{u-u_\lambda}^-(x), \ \text{if }  x\in \rN \setminus \Om_\lambda, 
\end{cases}
\end{align*}

$P_\lambda:= \operatorname{supp}v_\lambda \cap \Om_\lambda $ and $Q_\lambda := \operatorname{supp}v_\lambda \cap \bct{\rN\setminus \Om_\lambda}.$ Here, for any real number $a$, $a^+$ and $a^-$ represents $\max\{a,0\}$ and $\min\{a,0\}$ respectively. Clearly, for each $\lambda<0$ and $0\leq t<2s$
\begin{align}\label{crit_est}
\int_{\rN} \frac{|v_\lambda|^{2^*_t}}{|x'|^t} \leq \int_{P_\lambda} \frac{u^{2^*_t}}{|x'|^t} + \int_{Q_\lambda}  \frac{u_\lambda^{2^*_t}}{|x'|^t} = \int_{P_\lambda} \frac{u^{2^*_t}}{|x'|^t} + \int_{P_\lambda}  \frac{u^{2^*_t}}{|x'_\lambda|^t}<\infty,
\end{align}
where to get the finiteness we have used \eqref{asy_est}. However, it is not clear whether $v_\lambda$ belongs to $\mcal{\dot{H}}^s_\beta(\rN)$ or not. So, we will approximate $v_\lambda$ in a proper way.
 Let $\eta \in C^\infty(\rk)$ be such that $0\leq \eta \leq 1$ and  
\begin{align*}
\eta(x') = \begin{cases}
0, \ \text{if } |x'|<1 \\ 1 , \ \text{if } |x'|\geq 2.
\end{cases}
\end{align*}

For $\epsilon>0$, we define
\begin{align*}
\eta_{\epsilon}(x'):= \eta\bct{\frac{x'}{\epsilon}} \ \text{and } \eta_{\epsilon,\lambda}(x'): = \eta\bct{\frac{x'_\lambda}{\epsilon}}.
\end{align*}
For $h>0$, we also define $\psi_h(x): = \psi(x/h)$ and $\psi_{h,\lambda}(x):=\psi(x_\lambda/h) $, where $\psi\in C_c^\infty(\rN)$ such that $0\leq \psi\leq 1.$
\begin{align*}
\psi(x)= \begin{cases}
1, \ \text{if }|x|<1 \\ 0 , \ \text{if } |x|\geq 2.
\end{cases}
\end{align*}
We further define, $\phi_{\epsilon,h,\lambda}(x',x''):= \eta_{\epsilon}(x')\eta_{\epsilon,\lambda}(x')\psi_h(x)\psi_{h,\lambda}(x)$, $\Phi(x) : = \phi^2_{\epsilon,h,\lambda}(x)v_\lambda(x)$ and \\ $\tilde{\Phi}(x):= \phi_{\epsilon,h,\lambda}(x)v_{\lambda}(x)$. Then

\begin{align*}
E[u_\lambda,\Phi] &= \int_{\rN}\int_{\rN}\frac{\bct{u_\lambda(x)-u_\lambda(y)}\bct{\Phi(x)-\Phi(y)}}{\md{x-y}^{N+2s}}dx dy-\beta\int_{\rN}\frac{u_\lambda(x)\Phi(x)}{|x'|^{2s}} \\ &= \int_{\rN}\int_{\rN}\frac{\bct{u(x)-u(y)}\bct{\Phi_\lambda(x)-\Phi_\lambda(y)}}{\md{x-y}^{N+2s}}dx dy-\beta\int_{\rN}\frac{u(x)\Phi_\lambda(x)}{|x'|^{2s}} \\& +\beta\int_{\rN}u(x)\Phi_\lambda(x) \left[\frac{1}{|x'|^{2s}}-\frac{1}{\md{x'_\lambda}^{2s}}\right] dx,
\end{align*}
where to get the last equality we have used the fact that $\Phi\in C_c^{0,1}(\rN_k)$ supported away from $\partial \Om_{2\lambda}$. Now, using $\Phi_\lambda$ as a test function in \eqref{flap_cyl} we get
\begin{align*}
E[u_\lambda,\Phi]= \int_{\rN}\frac{u^{2^*_t-1}(x)\Phi_\lambda(x)}{|x'|^t} + \beta\int_{\rN}u(x)\Phi_\lambda(x) \left[\frac{1}{|x'|^{2s}}-\frac{1}{\md{x'_\lambda}^{2s}}\right] dx
\end{align*}



Observing that $v_\lambda$ is odd with respect to the reflection along $\partial\Om_\lambda$ and  
\begin{align*}
\left[\frac{1}{\md{x'}^{t}}-\frac{1}{\md{x'_\lambda}^{t}}\right]v_\lambda(x) \ \text{and } \left[\frac{1}{\md{x'}^{2s}}-\frac{1}{\md{x'_\lambda}^{2s}}\right]v_\lambda(x) 
\end{align*}
 both are nonpositive for all $x\in \rN$, we arrive at the following inequality
 \begin{align}\label{cyl_sym1}
 E[w_\lambda,\Phi] &\leq \int_{\rN} \bct{u^{2^*_t-1}-u_\lambda^{2^*_t-1}}(x)\Phi(x) \frac{dx}{|x'|^t} \notag\\ &\leq C \left[\int_{P_\lambda}u^{2^*_t-2}\frac{\tilde{\Phi}^2}{|x'|^t}dx + \int_{Q_\lambda}u_\lambda^{2^*_t-2}\frac{\tilde{\Phi}^2}{|x'|^t}dx\right], 
 \end{align}
 where $C>0$ is a generic constant. Let $\xi = \frac{2^*t}{2^*-2^*_t+2}.$ Clearly, $0\leq\xi <2s.$ Now, using H\"older inequality in \eqref{cyl_sym1} we get

 \begin{align}\label{cyl_sym2}
 E[w_\lambda,\Phi] & \leq C \bct{\int_{P_\lambda}u^{2^*}}^{\frac{2^*_t-2}{2^*}}\bct{\int_{P_\lambda}\frac{\md{\tilde{\Phi}}^{2^*_\xi}}{\md{x'}^\xi}}^{\frac{2}{2^*_\xi}} + C\bct{\int_{Q_\lambda}u_\lambda^{2^*}}^{\frac{2^*_t-2}{2^*}}\bct{\int_{Q_\lambda}\frac{\md{\tilde{\Phi}}^{2^*_\xi}}{\md{x'}^\xi}}^{\frac{2}{2^*_\xi}} \notag \\ &\leq C\bct{\int_{P_\lambda}u^{2^*}}^{\frac{2^*_t-2}{2^*}}\bct{\int_{\rN}\frac{\md{\tilde{\Phi}}^{2^*_\xi}}{\md{x'}^\xi}}^{\frac{2}{2^*_\xi}}.
 \end{align}
 Next, we will estimate $E[w_\lambda,\Phi]$ from below. For this we notice
 \begin{align}\label{cyl_sym3}
 &\bct{w_\lambda(x)-w_\lambda(y)}\bct{\Phi(x)-\Phi(y)}\notag\\  &= \bct{w_\lambda(x)-w_\lambda(y)} \bct{\phi_{\epsilon,h,\lambda}^2(x)v_\lambda(x)-\phi_{\epsilon,h,\lambda}^2(y)v_\lambda(y)} \notag\\&= \md{\tilde{\Phi}(x)-\tilde{\Phi}(y)}^2 \notag \\ &- \left[\phi_{\epsilon,h,\lambda}^2(x)v_{\lambda}(x)w_\lambda(y)+\phi_{\epsilon,h,\lambda}^2(y)v_\lambda(y)w_\lambda(x)-2\phi_{\epsilon,h,\lambda}(x)\phi_{\epsilon,h,\lambda}(y)v_\lambda(x)v_\lambda(y)\right]\notag\\ &\geq
 \begin{cases}
  \md{\tilde{\Phi}(x)-\tilde{\Phi}(y)}^2 - \phi_{\epsilon,h,\lambda}^2(y)v_\lambda(y)w_\lambda(x), \ \text{if } (x,y)\in P_\lambda^c\times\bct{P_\lambda\cup Q_\lambda}\cup Q_\lambda^c\times \bct{P_\lambda \cup Q_\lambda} \\ \md{\tilde{\Phi}(x)-\tilde{\Phi}(y)}^2 - \phi_{\epsilon,h,\lambda}^2(x)v_{\lambda}(x)w_\lambda(y), \ \text{if } (x,y)\in \bct{P_\lambda \cup Q_\lambda}\times P_\lambda^c \cup \bct{P_\lambda \cup Q_\lambda}\times Q_\lambda^c \\ \md{\tilde{\Phi}(x)-\tilde{\Phi}(y)}^2-\md{\phi_{\epsilon,h,\lambda}(x) -\phi_{\epsilon,h,\lambda}(y)}^2v_\lambda(x)v_{\lambda}(y) , \ \text{otherwise,}
 \end{cases}
 \end{align}
where $P_\lambda^c : = \Om_\lambda\setminus P_\lambda$ and $Q_\lambda^c:=\Om_\lambda^c\setminus Q_\lambda$. Since $\phi_{\epsilon,h,\lambda}$ symmetric, $v_\lambda$ and $w_\lambda$ antisymmetric with respect to the reflection along $\partial\Om_\lambda$, so 
considering the sign of $v_\lambda$ and $w_\lambda$ in respective region we have
 \begin{align*}
 &\int_{P_\lambda^c}\int_{P_\lambda} \phi_{\epsilon,h,\lambda}^2(y)v_\lambda(y)w_\lambda(x)\frac{dydx}{\md{x-y}^{N+2s}} + \int_{P_\lambda^c}\int_{Q_\lambda}  \phi_{\epsilon,h,\lambda}^2(y)v_\lambda(y)w_\lambda(x)\frac{dydx}{\md{x-y}^{N+2s}} \leq 0, \\& \int_{Q_\lambda^c}\int_{P_\lambda} \phi_{\epsilon,h,\lambda}^2(y)v_\lambda(y)w_\lambda(x)\frac{dydx}{\md{x-y}^{N+2s}} + \int_{Q_\lambda^c}\int_{Q_\lambda}  \phi_{\epsilon,h,\lambda}^2(y)v_\lambda(y)w_\lambda(x)\frac{dydx}{\md{x-y}^{N+2s}} \leq 0, \\
 &\int_{P_\lambda}\int_{P_\lambda^c} \phi_{\epsilon,h,\lambda}^2(x)v_\lambda(x)w_\lambda(y)\frac{dydx}{\md{x-y}^{N+2s}} + \int_{Q_\lambda}\int_{P_\lambda^c} \phi_{\epsilon,h,\lambda}^2(x)v_\lambda(x)w_\lambda(y)\frac{dydx}{\md{x-y}^{N+2s}} \leq 0, \\ &\int_{P_\lambda}\int_{Q_\lambda^c} \phi_{\epsilon,h,\lambda}^2(x)v_\lambda(x)w_\lambda(y)\frac{dydx}{\md{x-y}^{N+2s}} + \int_{Q_\lambda}\int_{Q_\lambda^c} \phi_{\epsilon,h,\lambda}^2(x)v_\lambda(x)w_\lambda(y)\frac{dydx}{\md{x-y}^{N+2s}} \leq 0.
 \end{align*}
Hence, integrating \eqref{cyl_sym3} we have 
\begin{align}\label{cyl_sym4}
E[w_\lambda,\Phi] &\geq \int_{\rN}\int_{\rN} \frac{\md{\tilde{\Phi}(x)-\tilde{\Phi}(y)}^2}{\md{x-y}^{N+2s}}dx dy- \beta \int_{\rN}\tilde{\Phi}^2(x)\frac{dx}{|x'|^{2s}} \notag\\ &-2 \int_{P_\lambda}\int_{P_\lambda}  \frac{v_\lambda(x)v_\lambda(y)\md{\phi_{\epsilon,h,\lambda}(x)-\phi_{\epsilon,h,\lambda}(y)}^2}{|x-y|^{N+2s}}dxdy
\end{align}
Now, consider
\begin{align}\label{cyl_sym5}
I_{\epsilon,h,\lambda}&:= \int_{P_\lambda}\int_{P_\lambda}  \frac{v_\lambda(x)v_\lambda(y)\md{\phi_{\epsilon,h,\lambda}(x)-\phi_{\epsilon,h,\lambda}(y)}^2}{|x-y|^{N+2s}}dxdy \notag\\ &\leq \int_{P_\lambda}\int_{P_\lambda}  \frac{u^2(x)\md{\phi_{\epsilon,h,\lambda}(x)-\phi_{\epsilon,h,\lambda}(y)}^2}{|x-y|^{N+2s}}dxdy \notag \\ &\leq 2\int_{\rN}\int_{\rN} \frac{u^2(x)\md{\psi_h(x)-\psi_h(y)}^2}{\md{x-y}^{N+2s}} + 2\int_{\rN}\int_{\rN} \frac{u_\lambda^2(x)\md{\psi_h(x)-\psi_h(y)}^2}{\md{x-y}^{N+2s}} \notag \\ &+ 2\int_{P_\lambda}\int_{P_\lambda} \frac{u^2(x)\psi^2_{h,\lambda}(y)\md{\eta_\epsilon(x')\eta_{\epsilon,\lambda}(x')-\eta_\epsilon(y')\eta_{\epsilon,\lambda}(y')}^2}{\md{x-y}^{N+2s}}dx dy.
\end{align}
By dividing $\rN$ in appropriate domain and using \eqref{asy_est} (which essentially shows that $u\in L^2(\rN)$, as $2\leq k\leq N-2$) one can estimate the first two term of 
\eqref{cyl_sym} to arrive at
\begin{align*}
&\int_{\rN}\int_{\rN} \frac{u^2(x)\md{\psi_h(x)-\psi_h(y)}^2}{\md{x-y}^{N+2s}} = \sigma_h \ \text{and } \\ &\int_{\rN}\int_{\rN} \frac{u_\lambda^2(x)\md{\psi_h(x)-\psi_h(y)}^2}{\md{x-y}^{N+2s}}= \sigma_{h,\lambda},
\end{align*}
 where for any $\lambda\in \ro$ both $\sigma_h$ and $\sigma_{h,\lambda}$ goes to zero as $h \rightarrow \infty$. To estimate the last term of \eqref{cyl_sym5} we use the fact that $u\in L^\infty(\Om_\lambda)$. Hence, finally we arrive at the following inequality
 \begin{align}\label{cyl_sym6}
 I_{\epsilon,h,\lambda} &\leq C_{N,s,t,\beta} \left[\sigma_h+\sigma_{h,\lambda}\right]\notag \\ &+ C_{N,s,t,\beta}\nrm{u}_{L^\infty(\Om_\lambda)}\int_{P_\lambda}\int_{P_\lambda} \frac{\psi^2_h(y)\md{\eta_\epsilon(x')\eta_{\epsilon,\lambda}(x')-\eta_\epsilon(y')\eta_{\epsilon,\lambda}(y')}^2}{\md{x-y}^{N+2s}}dx dy \notag \\ &\leq C_{N,k,s,t,\beta} \left[\sigma_h+\sigma_{h,\lambda}\right] \notag \\ &+h^{N-k}C_{N,k,s,t,\beta}\nrm{u}_{L^\infty(\Om_\lambda)} \int_{\rk}\int_{\rk} \frac{\md{\eta_\epsilon(x')-\eta_\epsilon(y')}^2}{\md{x'-y'}^{k+2s}} dx'dy'  \notag \\ &\leq C_{N,k,s,t,\beta} \left[\sigma_h+\sigma_{h,\lambda}+ h^{N-k}\epsilon^{k-2s}\nrm{u}_{L^\infty(\Om_\lambda)} \right].
 \end{align}
 Hence, combining \eqref{cyl_sym2}, \eqref{cyl_sym5}, \eqref{cyl_sym6} and using \eqref{FHSM_ineq} we conclude that there exist a constant $C_{N,k,s,t,\beta}>0$ depending on $N,k,s,t$ and $\beta$ such that the following holds for any $h>0$, $\epsilon>0$ and $\lambda<0$
 \begin{align*}
 \bct{\int_{\rN}\frac{\md{\tilde{\Phi}}^{2^*_\xi}}{\md{x'}^\xi}}^{\frac{2}{2^*_\xi}}&-\left[\sigma_h+\sigma_{h,\lambda} + h^{N-k}\epsilon^{k-2s}\nrm{u}_{L^\infty(\Om_\lambda)}\right] \\ &\leq C_{N,k,s,t,\beta}\bct{\int_{P_\lambda}u^{2^*}}^{\frac{2^*_t-2}{2^*}}\bct{\int_{\rN}\frac{\md{\tilde{\Phi}}^{2^*_\xi}}{\md{x'}^\xi}}^{\frac{2}{2^*_\xi}}
 \end{align*}
First letting $\epsilon \rightarrow 0$ then letting $h\rightarrow \infty$ and using DCT (because \eqref{crit_est} hold for each $\lambda<0$) we arrive at 
\begin{align}\label{cyl_sym7}
C_{N,k,s,t,\beta}\bct{\int_{\rN}\frac{\md{v_\lambda(x)}^{2^*_\xi}}{\md{x'}^\xi}}^{\frac{2}{2^*_\xi}} \leq \bct{\int_{P_\lambda}u^{2^*}}^{\frac{2^*_t-2}{2^*}}\bct{\int_{\rN}\frac{\md{v_\lambda(x)}^{2^*_\xi}}{\md{x'}^\xi}}^{\frac{2}{2^*_\xi}}
\end{align}

If $v_\lambda \neq 0$ a.e., then using \eqref{cyl_sym7}, we have that there exists a constant $C\equiv C(N,k,s,t,\beta)$ such that 
\begin{align}\label{cyl_sym8}
0<C(N,k,s,t,\beta) \leq \int_{P_\lambda}\md{u}^{2^*}, \ \text{for any } \lambda <0.
\end{align}
But, \eqref{cyl_sym6} yields a contradiction for large negative values of $\lambda $. So, for large negative values of $\lambda$ we must have $v_\lambda \equiv 0$ a.e. in  $\Om^i_\lambda.$ Hence, the set defined by
\begin{align*}
A := \{\lambda \leq 0 | u\leq u_\lambda \ \text{a.e in } \Om_\mu , \ \forall \mu \leq \lambda \}
\end{align*}
is non empty. Let $\bar{\lambda} := \sup A. $ For any $x'\in \rk$ and $R>0$, we also define the cylinder $C_R(x'):= B_R^k(x')\times \rNk.$
We claim that $\bar{\lambda} = 0$ and $u\leq u_{\bar{\lambda}}$ a.e. in $\Om_{\bar{\lambda}}$. If possible, let us assume that $\bar{\lambda}<0$. We define, $\tilde{w}_{\bar{\lambda}} := -w_{\bar{\lambda}}= u_{\bar{\lambda}}-u$. Let $\Om \subset \Om_{\bar{\lambda}}$. Since  $w_{\bar{\lambda}} \in L_s(\rN)\cap \dot{H}^s_\beta(\rN)$ and $u_{\bar{\lambda}}\geq u$ a.e. in $\Om_{\bar{\lambda}}$, $(-\Delta)^s \tilde{w}_{\bar{\lambda}} \geq 0$ in the distributional sense in $\Om.$ So, $w_{\bar{\lambda}}$ is lower semicontinuous as well as antisymmetric and a.e. nonnegative (by continuity) in $\Om_{\bar{\lambda}}.$ Hence, by Lemma \ref{mp_as} we have either $w_{\bar{\lambda}} =0$ a.e. in $\rN$ or $w_{\bar{\lambda}}>0$ in $\Om.$ We claim that, the second case can not occur. If it occurs, then $w_{\bar{\lambda}}>0$ and lower semicontinuous  in 
$\Om_{\bar{\lambda}}$ . 
Hence, by continuity of $u$, lower semicontinuity of $w_{\bar{\lambda}}$ and definition of $\bar{\lambda}$ we have for any $R_1>0$, $\delta>0$ small  and $R>0$ large, there exists $\epsilon_0(R_1,R,\delta)>0$ such that $\bar{\lambda}+\epsilon_0<0$ and 
\begin{align}
P_{\bar{\lambda}+\epsilon} \cap \Om_{\bar{\lambda}-\delta}\cap C_{R_1}^c(0) \subset \rN \setminus B^N_R(0), \ \forall 0<\epsilon \leq \epsilon_0.\notag
\end{align}
Now, since $\bar{\lambda}$ is the supremum so $v_{\bar{\lambda}+\epsilon}$ is not zero in a positive measure set, for any $0<\epsilon\leq\epsilon_0$. Hence, using \eqref{cyl_sym8} we have  \\
\begin{align*}
0<C(N,k,s,t,\beta) &\leq \int_{P_{\bar{\lambda}+\epsilon}} u^{2^*} \\ &\leq \int_{P_{\bar{\lambda}+\epsilon}\cap C_{R_1}(0)} u^{2^*} + \int_{P_{\bar{\lambda}+\epsilon}\cap C_{R_1}^c(0)} u^{2^*}  \\ &\leq \int_{C_{R_1}(0)} u^{2^*}  + \int_{P_{\bar{\lambda}+\epsilon}\cap \Om_{\bar{\lambda}}\cap C_{R_1}^c(0)} u^{2^*} + \int_{P_{\bar{\lambda}+\epsilon}\cap \Om^c_{\bar{\lambda}}\cap C^c_{R_1}(0)} u^{2^*} \\ &\leq \int_{C_{R_1}(0)} u^{2^*} +\int_{\rN\setminus B^N_R(0)} u^{2^*} +\int_{\Om_{\bar{\lambda}+\epsilon}\setminus \Om_{\bar{\lambda}-\delta}} u^{2^*}. 
\end{align*}
Now, first choosing $R_1$ small, $R$ large, $\delta$ small and then choosing $\epsilon$ small we can make the R.H.S. of the above inequality strictly less that $C(N,k,s,t,\beta)$, which gives a contradiction. 
So, either 
$\bar{\lambda}= 0$ or if $\bar{\lambda}<0$ then 
$w_{\bar{\lambda}} = 0$ a.e. in $\rN$. 
In the second case, \\
$\bct{-\Delta}^s w_{\bar{\lambda}} =0$ in 
$\rN$ in the distributional sense. But, since $\bar{\lambda}<0$ so, $\bct{-\Delta}^s w_{\bar{\lambda}}>0$ in $\Om$, in the distributional sense, for any open set $\Om \Subset \Om_{\bar{\lambda}}$.  Which gives a contradiction. Hence, $\bar{\lambda} = 0.$

Repeating the same arguments for $\lambda>0$, we  can conclude that $u$ is symmetric decreasing in $x_1$ direction.

\par \textbf{Step 2 : }In this step, we will show that there exist $\lambda_0 \in \ro$ such that, for any $x'\in\rk$ fixed, $u$ is symmetric w.r.t the reflexion along $\partial \Om'_{\lambda_0}.$ We will only prove an analogous inequality of \eqref{cyl_sym8} derived in Step 1. Rest of the arguments will be similar to Step1. We will exclude that part. We notice that in this case $u$ may not be in $L^\infty(\Om'_\lambda)$. Because of this we cannot use similar arguments of Step 1 to derive \eqref{cyl_sym8}. We define $w_\lambda(x) = u(x)-u(x_\lambda),$ where $x_\lambda:= (x',2\lambda-x_{k+1},x_{k+2},\dots,x_N)= (x',x''_\lambda).$ Then clearly, 

\begin{align*}
W_\lambda := \mcal{P}_{\tilde{\alpha}}[w_\lambda] := |x|^{\tilde{\alpha}} w_\lambda = U-U_\lambda,
\end{align*}
where we have denoted $\mcal{P}_{\tilde{\alpha}}[u]$ by $U$ and for any $0<\beta\leq \mcal{C}_{N,k,s}$, $\ta\in [0,\ta\leq \alpha]$ uniquely determined by \eqref{alphatilde}. Since, $u$ solves \eqref{flap_cyl} so $W_\lambda$ satisfies 

\begin{align}\label{cyl_sym9}
L^s_{\tilde{\alpha}}W_\lambda = A_\lambda W_\lambda \ \text{weakly in } \mcal{\dot{H}}^{s,\tilde{\alpha}}(\rN),
\end{align}
where
\begin{align*}
A_\lambda := \frac{1}{\md{x'}^{t+\tilde{\alpha}2^*_t}} \frac{U^{2^*_t-1}-U_\lambda^{2^*_t-1}}{U-U_\lambda}.
\end{align*}

Also, define  
\begin{align*}
v_\lambda(x) : = \begin{cases}
w_\lambda^+(x), \text{ if } x\in \Om'_\lambda \\ w^-_\lambda(x), \text{ if } x\in \bct{\Om'_\lambda}^c.
\end{cases}
\end{align*}
Then 
\begin{align*}
V_\lambda(x) : =\mcal{P}_{\tilde{\alpha}}[v_\lambda] = \begin{cases}
W_\lambda^+(x), \text{ if } x\in \Om'_\lambda \\ W^-_\lambda(x), \text{ if } x\in \bct{\Om'_\lambda}^c.
\end{cases}
\end{align*}
Similarly to the Step 1, we define 
\begin{align*}
P'_\lambda &:= \operatorname{supp}v_\lambda \cap \Om'_\lambda = \operatorname{supp}V_\lambda \cap \Om'_\lambda \\ Q'_\lambda &:= \operatorname{supp}v_\lambda \cap \bct{\Om'_\lambda}^c = \operatorname{supp}V_\lambda \cap \bct{\Om'_\lambda}^c.
\end{align*}
We define, $\phi_{\epsilon,h,\lambda}(x',x''):= \eta_{\epsilon}(x')\psi_h(x)\psi_{h,\lambda}(x)$, $\Phi(x) : = \phi^2_{\epsilon,h,\lambda}(x)V_\lambda(x)$ and \\ $\tilde{\Phi}(x):= \phi_{\epsilon,h,\lambda}(x)V_{\lambda}(x)$, where $\psi_h$ and $\psi_{h,\lambda}$ are same as defined in Step 1. Also, for $\epsilon<1$, $\eta_\epsilon\in C^{0,1}(\rk)$ and  satisfies the following
\begin{align*}
\eta_\epsilon(x')=\begin{cases} 0 , \text{ if } |x'|<\epsilon^2 \\ \frac{\ln\bct{\frac{|x'|}{\epsilon^2}}}{|\ln\epsilon|} , \text{ if } \epsilon^2\leq |x'|\leq \epsilon \\ 1 , \text{ if } |x'|>\epsilon.
\end{cases}
\end{align*}
Then clearly, $\Phi \in C_c^{0,1}(\rN_k)$ and so using $\Phi$ as a test function in \eqref{cyl_sym9} we get  

\begin{align}\label{cyl_sym11}
\langle L^s_{\tilde{\alpha}}W_\lambda, \Phi \rangle  = \int_{\rN} \tilde{\Phi}^2 A_\lambda(x) dx.
\end{align}
In the next few paragraphs $C$ will denote a positive constant possibly depending on  $N,s,k,t,\beta$. We estimate
\begin{align*}
\int_{\rN} \tilde{\Phi}^2 A_\lambda(x) dx &= \int_{\Om'_\lambda} \tilde{\Phi}^2(x) A_\lambda(x)dx + \int_{\rN\setminus \Om'_\lambda} \tilde{\Phi}(x) A_\lambda(x)dx \\ &\leq C\int_{P'_\lambda} \frac{U^{2^*_t-2}}{\md{x'}^{\tilde{\alpha}(2_t^*-2)}} \frac{\tilde{\Phi}^2}{\md{x'}^{t+\tilde{\alpha}(2^*_t-2)}} dx + C\int_{Q'_\lambda} \frac{U_\lambda^{2^*_t-2}}{\md{x'}^{\tilde{\alpha}(2_t^*-2)}} \frac{\tilde{\Phi}^2}{\md{x'}^{t+\tilde{\alpha}(2^*_t-2)}} dx,
\end{align*}
 Now using H\"older's inequality we have 
\begin{align}\label{cyl_sym12}
\int_{\rN} \tilde{\Phi}^2(x) A_\lambda(x) dx &\leq C\bct{\int_{P'_\lambda}U^{2^*}\frac{dx}{|x'|^{\tilde{\alpha}2^*}}}^{\frac{2^*_t-2}{2^*}}\bct{\int_{\rN}\frac{\md{\tilde{\Phi}}^{2^*_\xi}}{\md{x'}^{\xi+\tilde{\alpha}2^*_\xi}}}^{\frac{2}{2^*_\xi}} \notag \\&\leq C\bct{\int_{P'_\lambda}u^{2^*}}^{\frac{2^*_t-2}{2^*}}\bct{\int_{\rN}\frac{\md{\tilde{\Phi}}^{2^*_\xi}}{\md{x'}^{\xi+\tilde{\alpha}2^*_\xi}}}^{\frac{2}{2^*_\xi}},
\end{align}
where $\xi = \frac{t2^*}{2^*-2^*_t+2}.$ Next, we estimate  $\langle L^s_{\tilde{\alpha}}W_\lambda, \Phi\rangle$. Note that, proceeding similarly as in Step 1 we will arrive at 
\begin{align}\label{cyl_sym13}
\langle L^s_{\tilde{\alpha}}W_\lambda, \Phi\rangle &\geq \int_{\rN}\int_{\rN} \frac{\md{\tilde{\Phi}(x)-\tilde{\Phi}(y)}^2}{\md{x-y}^{N+2s} |x'|^{\tilde{\alpha}}|y'|^{\tilde{\alpha}}}dx dy \notag\\ &-2 \int_{P_\lambda}\int_{P_\lambda}  \frac{V_\lambda(x)V_\lambda(y)\md{\phi_{\epsilon,h,\lambda}(x)-\phi_{\epsilon,h,\lambda}(y)}^2}{|x-y|^{N+2s}|x'|^{\tilde{\alpha}}|y'|^{\tilde{\alpha}}}dxdy
\end{align}
Using $2ab\leq (a^2+b^2)$, with $a= \frac{V_\lambda(x)}{|x'|^{\ta}}$ and $b= \frac{V_\lambda(y)}{|y'|^{\ta}}$ whenever required, we estimate 
\begin{align}\label{cyl_sym14}
I_{\epsilon,h,\lambda}&:= \int_{P_\lambda}\int_{P_\lambda}  \frac{V_\lambda(x)V_\lambda(y)\md{\phi_{\epsilon,h,\lambda}(x)-\phi_{\epsilon,h,\lambda}(y)}^2}{|x-y|^{N+2s}|x'|^{\tilde{\alpha}}|y'|^{\tilde{\alpha}}}dxdy \notag\\ & \leq 2\int_{P_\lambda}\int_{P_\lambda}  \frac{V_\lambda(x)V_\lambda(y)\md{\psi_{h,\lambda}(x)-\psi_{h,\lambda}(y)}^2}{|x-y|^{N+2s}|x'|^{\tilde{\alpha}}|y'|^{\tilde{\alpha}}}dxdy  \notag \\ &+ 2\int_{P_\lambda}\int_{P_\lambda} \frac{V_\lambda(x)V_{\lambda}(y)\psi^2_{h,\lambda}(y)\md{\eta_\epsilon(x')-\eta_\epsilon(y')}^2}{\md{x-y}^{N+2s}|x'|^{\tilde{\alpha}}|y'|^{\tilde{\alpha}}}dx dy.
\notag \\ &\leq 2\int_{\rN}\int_{\rN} \frac{U^2(x)\md{\psi_h(x)-\psi_h(y)}^2}{|x'|^{2\tilde{\alpha}}\md{x-y}^{N+2s}} + 2\int_{\rN}\int_{\rN} \frac{U_\lambda^2(x)\md{\psi_h(x)-\psi_h(y)}^2}{|x'|^{2\tilde{\alpha}}\md{x-y}^{N+2s}} \notag \\ &+ 2\int_{P_\lambda}\int_{P_\lambda} \frac{V_\lambda(x)V_{\lambda}(y)\psi^2_{h,\lambda}(y)\md{\eta_\epsilon(x')-\eta_\epsilon(y')}^2}{\md{x-y}^{N+2s}|x'|^{\tilde{\alpha}}|y'|^{\tilde{\alpha}}}dx dy. \notag \\ &\leq 2\int_{\rN}\int_{\rN} \frac{u^2(x)\md{\psi_h(x)-\psi_h(y)}^2}{\md{x-y}^{N+2s}} + 2\int_{\rN}\int_{\rN} \frac{u_\lambda^2(x)\md{\psi_h(x)-\psi_h(y)}^2}{\md{x-y}^{N+2s}} \notag \\ &+ 2\nrm{U}^2_{L^\infty(\rN)}\int_{P_\lambda}\int_{P_\lambda} \frac{\psi^2_{h,\lambda}(y)\md{\eta_\epsilon(x')-\eta_\epsilon(y')}^2}{\md{x-y}^{N+2s}|x'|^{\tilde{\alpha}}|y'|^{\tilde{\alpha}}}dx dy,  
\end{align}
where in the last inequality we have used  Prposition \ref{linfty_est}. As remarked in the Step 1 we have
\begin{align*}
&\int_{\rN}\int_{\rN} \frac{u^2(x)\md{\psi_h(x)-\psi_h(y)}^2}{\md{x-y}^{N+2s}} = \sigma_h \ \text{and } \\ &\int_{\rN}\int_{\rN} \frac{u_\lambda^2(x)\md{\psi_h(x)-\psi_h(y)}^2}{\md{x-y}^{N+2s}}= \sigma_{h,\lambda},
\end{align*}
 where for any $\lambda\in \ro$ both $\sigma_h$ and $\sigma_{h,\lambda}$ goes to zero as $h \rightarrow \infty$.  Also, by Lemma \ref{log_cutoff} we conclude that 
\begin{align*}
\int_{P_\lambda}\int_{P_\lambda} \frac{\psi^2_{h,\lambda}(y)\md{\eta_\epsilon(x')-\eta_\epsilon(y')}^2}{\md{x-y}^{N+2s}|x'|^{\tilde{\alpha}}|y'|^{\tilde{\alpha}}}dx dy = \sigma_{\epsilon,h,\lambda},
\end{align*}
where for any $h>0$ and $\lambda \in \ro$, $\sigma_{\epsilon,h,\lambda}\rightarrow 0$ as $\epsilon \rightarrow 0$. So, form \eqref{cyl_sym14} we have \\ $I_{\epsilon,h,\lambda} \leq C\left(\sigma_h +\sigma_{h,\lambda}+ \sigma_{\epsilon,h,\lambda}\right).$

Hence, using \eqref{EFHSM_ineq}, \eqref{cyl_sym12} and \eqref{cyl_sym13} we have from \eqref{cyl_sym11} that there exist a constant $C_{N,k,s,t,\beta}>0$ depending on the indexed variables such that 
\begin{align*}
C_{N,k,s,t,\beta}\bct{\int_{\rN}\frac{\md{\tilde{\Phi}(x)}^{2^*_\xi}}{\md{x'}^{\xi+\tilde{\alpha}2^*_\xi}}dx}^{\frac{2}{2^*_\xi}} &- C\left(\sigma_h +\sigma_{h,\lambda}+ \sigma_{\epsilon,h,\lambda}\right) \\&\leq \bct{\int_{P'_\lambda}u^{2^*}}^{\frac{2^*_t-2}{2^*}}\bct{\int_{\rN}\frac{\md{\tilde{\Phi}(x)}^{2^*_\xi}}{\md{x'}^{\xi+\tilde{\alpha}2^*_\xi}}dx}^{\frac{2}{2^*_\xi}}.
\end{align*}
First letting $\epsilon \rightarrow 0$ then letting $h\rightarrow \infty$ and using DCT  we arrive at 
\begin{align*}
C_{N,k,s,t,\beta}\bct{\int_{\rN}\frac{\md{V_\lambda(x)}^{2^*_\xi}}{\md{x'}^{\xi+\tilde{\alpha}2^*_\xi}}}^{\frac{2}{2^*_\xi}} \leq \bct{\int_{P_\lambda}u^{2^*}}^{\frac{2^*_t-2}{2^*}}\bct{\int_{\rN}\frac{\md{V_\lambda(x)}^{2^*_\xi}}{\md{x'}^{\xi+\tilde{\alpha}2^*_\xi}}}^{\frac{2}{2^*_{\xi}}}
\end{align*}

Hence, for $V_\lambda \neq 0$ a.e. we have 
\begin{align*}
0<C_{N,k,s,t,\beta} \leq \int_{P'_\lambda}\md{u}^{2^*}, \ \text{for any } \lambda <0,
\end{align*}
which is the exact counterpart of inequality \eqref{cyl_sym8}.

\par Combining Step 1 and Step 2 we conclude that $u$ is cylindrically symmetric.
\end{proof}

\section{Appendix}  

\subsection{A Density Property} Main aim of this section is to derive Lemma \ref{rep_sob_sp}. The arguments are modifications of those in \cite{DV}, where the Muckenhoupt $A_1$ properties of the weights have been used crucially. We will sketch the proof by pointing out main steps. First, let us define
\begin{align*}
\mcal{W} := \{u\in L^{2^*}\left(\rN;\frac{1}{|x'|^{\tilde{\alpha}2^*}}\right) : \int_{\rN}\int_{\rN}\frac{\md{u(x)-u(y)}^2 dxdy}{\md{x-y}^{N+2s}|x'|^{\tilde{\alpha}}|y'|^{\tilde{\alpha}}}< \infty \},
\end{align*}
endowed with the following norm
\begin{align*}
\nrm{u}_{\mcal{W}} := \left[\left[ u\right]\right]_{s,\ta,\rN} + \nrm{u}_{2^*,\ta,\rN} .
\end{align*}
Here,
\begin{align*}
\nrm{u}_{2^*,\ta,\rN} : = \bct{\int_{\rN} \frac{\md{u(x)}^{2^*}}{|x'|^{\ta2^*}}dx}^{\frac{1}{2^*}},
\end{align*}
and the semi-norm $\left[\left[u\right]\right]_{s,\ta,\rN}$ is same, as defined in Section \ref{Prem}. We also, define the following:
\begin{align}\label{weights}
\text{when } &N' = 2N, \ w(z,z) = (z,z), \ \Theta(X) =|x'|^{\ta}|y'|^{\ta}, \ X=(x,y); \ x,y,z\in \rN,  \notag \\ \text{and when } &N' = N, \ w(z)= z,\ \Theta(X) = |x'|^{\ta2^*},\ X=x ;\ x,y\in \rN. 
\end{align}
Next, we will prove the following lemma.

\begin{lem}\label{log_cutoff}
Let $u\in C_c^\infty(\rN)$. Also, we consider $\eta_\epsilon$, defined by the following
\begin{align*}
\eta_\epsilon(x')=\begin{cases} 0 , \text{ if } |x'|<\epsilon^2 \\ \frac{\ln\bct{\frac{|x'|}{\epsilon^2}}}{|\ln\epsilon|} , \text{ if } \epsilon^2\leq |x'|\leq \epsilon \\ 1 , \text{ if } |x'|>\epsilon.
\end{cases}
\end{align*}
 Then, for any $0<\tilde{\alpha}\leq (k-2s)/2$, the following are true 
 \begin{itemize}
 \item[(i)] $\int_{\rN}\int_{\rN} \frac{\md{u(x)-u(y)}^2}{\md{x-y}^{N+2s}|x'|^{\tilde{\alpha}}|y'|^{\tilde{\alpha}}} dxdy<\infty ,$
 \item[(ii)] $\displaystyle\lim_{\epsilon\rightarrow 0} \int_{\rN}\int_{\rN} \frac{u^2(x)\md{\eta_\epsilon(x')-\eta_\epsilon(y')}^2}{\md{x-y}^{N+2s}|x'|^{\tilde{\alpha}}|y'|^{\tilde{\alpha}}}dxdy =0$.
 \end{itemize}
In particular, $\eta_\epsilon u\in C_c^{0,1}(\rN_k)$ converges to $u$ under the semi norm $\left[[.]\right]_{s,\tilde{\alpha},\rN}$, i.e $u\in \mcal{\dot{H}}^{s,\tilde{\alpha}}(\rN)$.
\end{lem}
\begin{proof}
We will only prove (ii). One can easily check that (i) holds in fact for $u\in C_c^{0,1}(\rN)$. Notice that 
\begin{align}\label{log_cutoff1}
&\int_{\rN}\int_{\rN} \frac{u^2(x)\md{\eta_\epsilon(x')-\eta_\epsilon(y')}^2}{\md{x-y}^{N+2s}|x'|^{\tilde{\alpha}}|y'|^{\tilde{\alpha}}}dxdy \notag \\  &= \int_{\rk}\int_{\rk} \frac{\md{\eta_\epsilon(x')-\eta_\epsilon(y')}^2}{|x'|^{\tilde{\alpha}}|y'|^{\tilde{\alpha}}}dx'dy' \int_{\rNk} u^2(x',x'') \int_{\rNk} \frac{dy''}{\bct{|x'-y'|^2+|x''-y''|^2}^{\frac{N+2s}{2}}} dx'' \notag \\ &\leq  C \int_{\rk}\int_{\rk} \frac{\md{\eta_\epsilon(x')-\eta_\epsilon(y')}^2}{\md{x'-y'}^{k+2s}|x'|^{\tilde{\alpha}}|y'|^{\tilde{\alpha}}}dx'dy', 
\end{align}
 where (and for the rest of the proof) $C>0$ is constant depending on $N,k,s, \tilde{\alpha}$, $\nrm{u}_{L^\infty(\rN)}$ and $\operatorname{supp}u$.  We define
 \begin{align*}
 I_\epsilon(x',y') :&=  \frac{\md{\eta_\epsilon(x')-\eta_\epsilon(y')}^2}{\md{x'-y'}^{k+2s}|x'|^{\tilde{\alpha}}|y'|^{\tilde{\alpha}}}dx'dy'  \text{ and } \\
 H_\epsilon :&= \int_{\rN}\int_{\rN} I_\epsilon(x',y') dx'dy'.
 \end{align*}
Then, in view of the \eqref{log_cutoff1}, it is enough to show $H_\epsilon = o(1)$ as $\epsilon\rightarrow0.$ We define 
\begin{align*}
H_{\epsilon,1}&:= \int_{|x'|<\epsilon^2}\int_{\epsilon^2<|y'|<\epsilon}I_\epsilon, \ H_{\epsilon,2} : = \int_{|x'|>\epsilon}\int_{\epsilon^2<|y'|<\epsilon} I_\epsilon, \\ H_{\epsilon,3}  &:= \int_{|x'|<\epsilon^2} \int_{|y'|>\epsilon}, \ H_{\epsilon,4} : = \int_{\epsilon^2<|x'|<\epsilon}\int_{\epsilon^2<|y'|<\epsilon} I_\epsilon.
\end{align*}
Then using the symmetry of $I_\epsilon$ we have 
\begin{align*}
H_\epsilon: = 2H_{\epsilon,1}+ 2H_{\epsilon,2}+2H_{\epsilon,3} +H_{\epsilon,4}.
\end{align*}
We will show $H_{\epsilon,m} = o(1)$ as $\epsilon \rightarrow 0$ for $m=1,2,3,4$.
\par \textbf{Step 1:} In this step will estimate $H_{\epsilon,1}.$ For this we define 
\begin{align*}
F_{x'}: &=\{\epsilon^2<|y'|<\epsilon\} \cap \{y': |y'-x'|\geq \epsilon^2/2\} , \text{ and } \\ F'_{x'} : &= \{\epsilon^2<|y'|<\epsilon\} \cap \{y': |y'-x'|< \epsilon^2/2\}. \text{ Then }\\ 
H_{\epsilon,1} &= \int_{|x'|<\epsilon^2} \int_{F_{x'}}I_\epsilon + \int_{|x'|<\epsilon^2}\int_{F'_{x'}} I_\epsilon.
\end{align*}
We first consider 
\begin{align}\label{log_cutoff2}
\int_{|x'|<\epsilon^2}\int_{F_{x'}'}I_\epsilon &\leq \frac{C}{|\ln\epsilon|^2} \int_{|x'|<\epsilon^2} \frac{1}{|x'|^{2\tilde{\alpha}}} \int_{F'_{x'}} \frac{\md{\ln|y'|-\ln|x'|}^2}{|x'-y'|^{k+2s}} dy' dx' \notag\\ &\leq\int_0^1 \frac{C}{|\ln\epsilon|^2} \int_{|x'|<\epsilon^2} \frac{1}{|x'|^{2\tilde{\alpha}}} \int_{F'_{x'}} \frac{dy'}{|x'-y'|^{k+2s-2}\md{y'+r(x'-y')}^2}  dx' dr \notag\\ &\leq \frac{C}{|\ln\epsilon|^2} \int_{|x'|<\epsilon^2} \frac{1}{|x'|^{2\tilde{\alpha}}} \int_{F'_{x'}} \frac{dy'}{|x'-y'|^{k+2s-2}|y'|^2} dx' \notag \\ &\leq \frac{C}{\epsilon^4|\ln\epsilon|^2} \int_{|x'|<\epsilon^2} \frac{1}{|x'|^{2\tilde{\alpha}}} \int_{\{|x'-y'|<\frac{\epsilon^2}{2}\}} \frac{dy'}{|x'-y'|^{k+2s-2}} dx' \notag  \\ &\leq \frac{C}{\epsilon^4|\ln\epsilon|^2} \epsilon^{4s} \epsilon^{4-4s} = o(1),  \text{ as } \epsilon \rightarrow 0 ,
\end{align}
where in the last inequality we have used the fact that, for small \\$\epsilon>0$, $\epsilon^{k-2\tilde{\alpha}} \leq \epsilon^{4s}$, for any $0<2\tilde{\alpha}\leq k-2s$.   
Next, we consider
\begin{align}\label{log_cutoff3}
\int_{|x'|<\epsilon^2} \int_{F_{x'}} I_\epsilon &\leq \frac{C}{|\ln\epsilon|^2} \int_{|x'|\leq \epsilon^2}\frac{1}{|x'|^{\tilde{\alpha}}} \int_{F_{x'}} \frac{\ln^2(\frac{|y'|}{\epsilon^2})}{\md{x'-y'}^{k+2s}|y'|^{\tilde{\alpha}}}dy' dx' \notag \\ &\leq \frac{C\epsilon^{k-2s-2\tilde{\alpha}}}{|\ln\epsilon|^2} \int_{|x'|\leq 1}\frac{1}{|x'|^{\tilde{\alpha}}} \int_{\substack{\{1<|y'|<\frac{1}{\epsilon}\} \\ \cap \{|x'-y'|>\frac{1}{2}\}}} \frac{\ln^2|y'|}{\md{x'-y'}^{k+2s}|y'|^{\tilde{\alpha}}}dy' dx' \notag \\ &\leq o(1)+ \frac{C}{|\ln\epsilon|^2} \int_{|x'|<1} \int_{\substack{\{2<|y'|<\frac{1}{\epsilon}\} \\ \cap \{|x'-y'|>\frac{1}{2}\}}}\frac{\ln^2|y'|}{\md{x'-y'}^{k+2s}|y'|^{\tilde{\alpha}}}dy' dx'\notag \\ &\leq o(1) +\frac{C}{|\ln\epsilon|^2}\int_{\{2<|y'|<\epsilon^{-1}\}} \frac{\ln^2|y'|dy'}{|y'|^{k+2s+\frac{k-2s}{2}}} \notag \\ 
&\leq o(1) +\frac{C}{|\ln\epsilon|^2}\int_{\{2<|y'|<\epsilon^{-1}\}} \frac{\ln^2|y'|dy'}{|y'|^{k+2s}}  = o(1), \text{ as } \epsilon \rightarrow 0.
\end{align}
Hence, combining \eqref{log_cutoff2} and \eqref{log_cutoff3} we have $H_{\epsilon,1} = o(1),$ as $\epsilon \rightarrow 0$.
\par \textbf{Step 2:} In this step, we will show that $H_{\epsilon,m} =o(1) $, as $\epsilon \rightarrow 0$ for $m=2,3$. In fact, we will show this, only for the case $m=2$. The assertion, for the case, $m=3$, will follow similarly and much more easily.  \par By a change of variable we get 
\begin{align}\label{log_cutoff4}
H_{\epsilon,2} &\leq  \frac{1}{|\ln\epsilon|^2} \int_{\epsilon<|y'|<1}\frac{1}{|y'|^{2\tilde{\alpha}}}\int_{|x'|>1}  \frac{\md{\ln\frac{|y'|}{\epsilon}-\ln\frac{1}{\epsilon}}^2}{|x'-y'|^{k+2s}} dy'dx' \notag \\& \leq H'_{\epsilon,2}+H''_{\epsilon,2},
\end{align}
where 
\begin{align*}
H'_{\epsilon,2} : &= \frac{C}{|\ln\epsilon|^2} \int_0^1\int_{\epsilon<|y'|<1} \frac{1}{|y'|^{2\tilde{\alpha}}} \int_{\{|x'|>1\}\cap \{|x'-y'|\leq \frac{1}{2}\}} \frac{dx'}{|x'-y'|^{k+2s-2}|x'+r(x'-y')|^2}dy'dr \\ &\leq \frac{C}{|\ln\epsilon|^2} \int_{\epsilon<|y'|<1} \frac{1}{|y'|^{2\tilde{\alpha}}}\int_{\{|x'-y'|\leq \frac{1}{2}\}} \frac{dx'}{|x'-y'|^{k+2s-2}}dy' = o(1),  \text{ as } \epsilon \rightarrow 0,
\end{align*}
and 
\begin{align*}
H''_{\epsilon,2} : &= \frac{C}{|\ln\epsilon|^2} \int_{\epsilon<|y'|<1}\frac{1}{|y'|^{2\tilde{\alpha}}}\int_{\{|x'|>1\}\cap \{|x'-y'|\geq \frac{1}{2}\}}  \frac{\md{\ln\frac{|y'|}{\epsilon}-\ln\frac{1}{\epsilon}}^2}{|x'-y'|^{k+2s}} dy'dx' \\ &\leq  \frac{C}{|\ln\epsilon|^2} \int_{\epsilon<|y'|<1}  \frac{\ln^2|y'|}{|y'|^{2\tilde{\alpha}}}dy' = \frac{C}{|\ln\epsilon|^2}\int_\epsilon^1 \frac{\ln^2r}{r^{1-\alpha'}} dr \\ &= o(1), \text{ as } \epsilon \rightarrow 0,
\end{align*}
where $\alpha' = k-2\tilde{\alpha}\geq 2s.$ Hence, from \eqref{log_cutoff4} we have $H_{\epsilon,2}= o(1),$ as $\epsilon \rightarrow 0$. 
\par \textbf{Step 3:} In this step, we will show that $H_{\epsilon,4} =o(1) $, as $\epsilon \rightarrow 0$. Similarly, considering different regions, we see that, it is enough to show the following:
\begin{align}\label{log_cutoff5}
H_{\epsilon,4,1}: = \frac{1}{|\ln \epsilon|^2} \int\int\limits_{F}\frac{\md{\ln|x'|-\ln|y'|}^2}{\md{x'-y'}^{N+2s}|x'|^{\tilde{\alpha}}|y'|^{\tilde{\alpha}}}dx'dy' = o(1), \text{ as } \epsilon \rightarrow 0,
\end{align}
  where $F$ is defined as follows:
  \begin{align*}
  F: = \{(x',y'): \epsilon^2 < |x'|\leq |y'| <\epsilon \text{ and } |y'|<2|x'|\}.
  \end{align*}
Clearly, $F\subset \{(x',y'): \epsilon^2 < |x'|\leq |y'| <\epsilon \text{ and } |x'- y'|<3|x'|\}$. So, using, $\ln r \leq r-1$, for $r\geq 1$, we estimate
\begin{align*}
H_{\epsilon,4,1} &\leq \frac{1}{|\ln\epsilon|^2}  \int_{\epsilon^2<|x'|<\epsilon}\frac{1}{|x'|^{k-2s+2}}  \int_{|x'-y'|<3|x'|} \frac{dy'}{|x'-y'|^{k+2s-2}}dx' \\ &= \frac{C}{|\ln\epsilon|^2}\int_{\epsilon^2<|x'|<\epsilon} \frac{dx'}{|x'|^k} = o(1), \text{ as } \epsilon\rightarrow 0.
\end{align*}

Combining Step 1, Step 2 and Step 3 we conclude the lemma.

\end{proof}
In light of the Lemma \ref{log_cutoff}, it is enough to prove that $C_c^\infty(\rN)$ is dense in $\mcal{W}$ to conclude Lemma \ref{rep_sob_sp}. The following Lemma shows that, we can approximate $u\in \mcal{W}$ by a sequence of compactly supported functions lying in $\mcal{W}$.

\begin{lem}\label{cpt_apprx}
Let $u\in \mcal{W}$, $0<\tilde{\alpha}\leq \frac{k-2s}{2}$ and $\eta \in C_c^\infty\left(B_2^N(0);[0,1]\right)$ such that $\eta= 1$ in $B_1^N(0)$ and $\eta_j(x) = \eta(x/j)$. Then 
\begin{align*}
\displaystyle{\lim_{j\rightarrow \infty}} \left[\nrm{u-\eta_ju}_{2^*,\ta,\rN} + \left[\left[u-\eta_ju\right]\right]_{s,\ta,\rN}\right] =0.
\end{align*}

\end{lem}
\begin{proof}
We define
\begin{align*}
I_j : = \int_{\rN}\int_{\rN\setminus B_j^N(0)} \frac{|u(y)|^2\md{\eta_j(x)-\eta_j(y)}^2}{|x-y|^{N+2s}|x'|^{\tilde{\alpha}}|y'|^{\tilde{\alpha}}} dy dx.
\end{align*}
Since, $\eta = 1$ on $B_1^N(0),$ so, to prove the Lemma, it is enough to prove, $\displaystyle{\lim_{j\rightarrow \infty}}I_j = 0$. We define the following sets
\begin{align*}
D_{j,0}&: = \{(x,y) \in \rN\times (\rN\setminus B_j^N(0)) : |x|\leq |y|/2\},\\ D_{j,1}&: = \{(x,y) \in \rN\times (\rN\setminus B_j^N(0)) : |x|\geq |y|/2 \text{ and } |x-y|\geq j\},  \\  D_{j,2}&: = \{(x,y) \in \rN\times (\rN\setminus B_j^N(0)) : |x|\geq |y|/2 \text{ and } |x-y|\leq j\} . 
\end{align*}
For $m=0,1,2$, we write 
\begin{align*}
I_{j,m} : = \int\int_{D_{j,m}} \frac{|u(y)|^2\md{\eta_j(x)-\eta_j(y)}^2}{|x-y|^{N+2s}|x'|^{\tilde{\alpha}}|y'|^{\tilde{\alpha}}} dy dx.
\end{align*}
Then
\begin{align}\label{cpt_apprx1}
 I_j = I_{j,0}+I_{j,1}+I_{j,2}.
\end{align}
Now, we break
\begin{align*}
\frac{|u(y)|^2\md{\eta_j(x)-\eta_j(y)}^2}{|x-y|^{N+2s}|x'|^{\tilde{\alpha}}|y'|^{\tilde{\alpha}}} = \frac{|\eta_j(x)-\eta_j(y)|^2}{|x-y|^{2s+2\sigma_m}} \frac{|u(y)|^2}{|x-y|^{N-2\sigma_m|x'|^{\tilde{\alpha}}|y'|^{\tilde{\alpha}}}},
\end{align*}
where $\sigma_0 = s,$ $\sigma_1 \in (0,s)$ and $s<\sigma_2<1$ such that $\frac{N(N-2\sigma_2)}{N-2s}>\max\{N-k,k\}.$ We denote, $\sigma'_m := \frac{N-2\sigma_2}{N-2s}$. 
Then using H\"older inequality
\begin{align}\label{cpt_apprx2}
I_{j,m} \leq \bct{\int\int_{D_{j,m}} \frac{|\eta_j(x)-\eta_j(y)|^{\frac{N}{s}}}{\md{x-y}^{N+\sigma_m\frac{N}{s}}}dxdy}^{\frac{2s}{N}} \bct{\int\int_{D_{j,m}} \frac{|u(y)|^{2^*} }{|x-y|^{N\sigma'_m }|x'|^{\tilde{\alpha}} |y'|^{\tilde{\alpha}}}dx dy } ^{\frac{N-2s}{N} }.
\end{align}
Clearly, 
\begin{align}\label{cpt_apprx3}
\int\int_{D_{j,m}} \frac{|\eta_j(x)-\eta_j(y)|^{\frac{N}{s}}}{\md{x-y}^{N+\sigma_m\frac{N}{s}}}dxdy \leq j^{(s-\sigma_m)\frac{N}{s}} \int_{\rN}\int_{\rN} \frac{|\eta_j(x)-\eta_j(y)|^{\frac{N}{s}}}{\md{x-y}^{N+\sigma_m\frac{N}{s}}}dxdy \leq C j^{\frac{(s-\sigma_m)N}{s}}.
\end{align}
Now, we consider 
\begin{align}\label{cpt_apprx4}
\int\int_{D_{j,0}} \frac{|u(y)|^{2^*} }{|x-y|^{N\sigma'_0 }|x'|^{\tilde{\alpha}} |y'|^{\tilde{\alpha}}}dx dy  &\leq \int_{|y|>j} \frac{|u(y)|^{2^*}}{|y'|^\frac{2^*{\tilde{\alpha}}}{2}} \int_{|x|<\frac{|y|}{2} } \frac{dx}{|x-y|^N|y'|^\frac{2^*{\tilde{\alpha}}}{2}} dy\notag \\ &\leq C\int_{|y|>j} \frac{|u(y)|^{2^*}}{|y'|^{\frac{2^*{\tilde{\alpha}}}{2}}|y|^{\frac{{\tilde{\alpha}}2^*}{2}}}dy \leq C \int_{|y|>j} \frac{|u(y)|^{2^*}}{|y'|^{\alpha2^*}}dy.
\end{align}
Since, $N\sigma'_1>N>N-k$, we estimate
\begin{align}\label{cpt_apprx5}
\int\int_{D_{j,1}} \frac{|u(y)|^{2^*} }{|x-y|^{N\sigma'_1 }|x'|^{\tilde{\alpha}} |y'|^{\tilde{\alpha}}}dx dy &\leq \int_{|y|>j} \int_{\substack{\{|x-y|>j\} \\ \cap \{|x'|\leq \frac{|y'|}{2} \}}} \frac{|u(y)|^{2^*} }{|x-y|^{N\sigma'_1 }|x'|^{\tilde{\alpha}} |y'|^{\tilde{\alpha}}}dx dy\notag \\ &+ \int_{|y|>j} \int_{\substack{\{|x-y|>j\} \\ \cap \{|x'|\geq \frac{|y'|}{2}\}}}\frac{|u(y)|^{2^*} }{|x-y|^{N\sigma'_1 }|x'|^{\tilde{\alpha}} |y'|^{\tilde{\alpha}}}dx dy \notag\\  &\leq C \int_{|y|>j} \frac{|u(y)|^{2^*}}{|y'|^{\frac{{\tilde{\alpha}}2^*}{2}}} \int_{|x'|<\frac{|y'|}{2}} \frac{dx'}{|x'-y'|^{N\sigma'_1-N+k} |x'|^{\frac{{\tilde{\alpha}}2^*}{2}}} dy' \notag \\ &+ C\int_{|y|>j}\frac{|u(y)|^{2^*}}{|y'|^{2^*{\tilde{\alpha}}}} \int_{|x-y|>j} \frac{dx}{|x-y|^{N\sigma'_1}} dy \notag \\ &\leq C\int_{|y|>j} \frac{|u(y)|^{2^*}}{|y'|^{{\tilde{\alpha}}2^*}} \frac{dy}{|y'|^{N\sigma'_1-N}} \notag \\&+ C\frac{1}{j^{N\sigma'_1-1}}\int_{|y|>j} \frac{|u(y)|^{2^*}}{|y'|^{2^*{\tilde{\alpha}}}}dy \notag \\ &\leq \frac{C}{j^{\frac{2N(s-\sigma_1)}{N-2s}}} \int_{|y|>j} \frac{|u(y)|^{2^*}}{|y'|^{2^*{\tilde{\alpha}}}}dy.
\end{align}
Similarly, using $N>N\bar{\sigma}_2> \max\{ N-k,k\}$, we can derive
\begin{align}\label{cpt_apprx6}
\int\int_{D_{j,2}} \frac{|u(y)|^{2^*} }{|x-y|^{N\sigma'_1 }|x'|^{\tilde{\alpha}} |y'|^{\tilde{\alpha}}}dx dy \leq  \frac{C}{j^{\frac{2N(s-\sigma_2)}{N-2s}}} \int_{|y|>j} \frac{|u(y)|^{2^*}}{|y'|^{2^*{\tilde{\alpha}}}}dy.
\end{align}

Hence, plugging \eqref{cpt_apprx3}, \eqref{cpt_apprx4}, \eqref{cpt_apprx5} and \eqref{cpt_apprx6} into \eqref{cpt_apprx2} and then using  \eqref{cpt_apprx1} we get 
\begin{align*}
I_j \leq C \nrm{u}^2_{L^{2^*}\left(\rN\setminus B_j^N(0); \frac{1}{|x'|^{2{\tilde{\alpha}}}}\right)} \rightarrow 0 , \text{ as } j \rightarrow \infty.
\end{align*}
This proves the lemma.
\end{proof}

The next proposition is a reminiscence of the fact, that $\Theta$ is in $A_1$. Although, in this case, the proof is a direct consequence of Proposition $4.1$ and $4.2$ of \cite{DV}.
\begin{prop}\label{a1property}
There exists a constant $C>0$ such that for every $X\in \rN_k\times\rN_k$, when $N'=2N$ and $X\in \rN_k$, when $N'=N$, the following inequality is true 
\begin{align*}
\displaystyle{\sup_{r>0}}\frac{1}{r^N}\int_{B^N_r(0)}\frac{dz}{\Theta\left(X+w(z)\right)} \leq \frac{C}{\Theta(X)}.
\end{align*}
\end{prop}
Using Proposition \ref{a1property} and the fact, that the measure $\frac{dX}{\Theta(X)}$, is finite on compact sets of $\ro^{N'}$, we can derive the following lemma which is related to the boundedness of the maximal operator.
\begin{lem}\label{max_op}
Let $q>1$ and $V:\ro^{N'}\rightarrow \ro$ be a measurable function. Then, for any $r>0$,
\begin{align*}
\int_{\ro^{N'}} \bct{\frac{1}{r^n}\int_{B_r^{N'}(0)} \md{V(X-w(z))}dz}^q\frac{dX}{\Theta(X)} \leq C \int_{\ro^{N'}} \frac{|V(X)|^q}{\Theta(X)},
\end{align*}
for some constant $C>0$.
\end{lem}
Next, for $V:\ro^{N'}:\rightarrow \ro$ measurable, we define the following operator
\begin{align*}
V\star \eta_0 (X) :=  \int_{\rN} V(X-w(z))\eta_0(z)dz, 
\end{align*}
 where $\eta_0$ is a radially symmetric mollifier in $\rN$, with $\eta_0\geq 0$ and $\operatorname{supp}\eta_0 \subset B_1^N(0)$. Notice that, when $N'= N$, $V\star \eta_0$ coincides with the usual convolution operator $V*\eta_0$. As a consequence of Lemma \ref{max_op}, we could control appropriate weighted $L^p$ norm of $V\star \eta_0$. More precisely, we could derive the following proposition.
 \begin{prop}\label{conv_contrl}
There exists a constant $C>0$, such that for any measurable function $V: \ro^{N'}\rightarrow \ro$ we have
\begin{align*}
\int_{\ro^{N'}} \md{V\star \eta_0}^p \frac{dX}{\Theta(X)} \leq C\int_{\ro^{N'}}|V(X)|^p\frac{dX}{\Theta(X)},
\end{align*}
where $p=2$, when $N'=2N$ and $p=2^*$, when $N'=N$.
 \end{prop}
\subsection{Proof Of Lemma \ref{rep_sob_sp}}
\begin{proof}
We define 
\begin{align*}
L^p(\ro^{N'};\Theta) := \{V:\ro^{N'}\rightarrow \ro \text{ measurable }: \int_{\ro^{N'}} |V(X)|^p \frac{dX}{\Theta(X)} <\infty\},
\end{align*}
where $p$ is defined in the Proposition \ref{conv_contrl}. Then, since $\frac{dX}{\Theta(X)}$ is finite on compact sets of $\ro^{N'}$, so using Lusin' s theorem and Proposition \ref{conv_contrl},  we can prove that $C_c^\infty(\ro^{N'})$ is dense in $L^p(\ro^{N};\Theta)$. As a consequence of this density and Proposition \ref{cpt_apprx} and the fact, that for any $u\in \mcal{W}$ and $\eta \in C_c^\infty(\rN)$, $V^u*\eta = V^{u*\eta}$, we can prove that $C_c^\infty(\rN)$ is dense in $\mcal{W}$, where $V^u(x,y): = \frac{u(x)-u(y)}{\md{x-y}^{\frac{N}{2}+s}}$, for $x,y \in \rN$. This proves that, $\mcal{\dot{H}}^{s,\ta}(\rN) = \mcal{W},$ which is exactly what we wanted to prove in Lemma \ref{rep_sob_sp}.
\end{proof}

\section*{Acknowledgments}
I would like to thank my Ph.D supervisor Prof. K. Sandeep for countless valuable discussions and suggestions.

\def\cprime{$'$}





\begin{thebibliography}{10}

\bibitem{AbMePe}
Boumediene Abdellaoui, Mar\'ia Medina, Ireneo Peral, and Ana Primo.
\newblock The effect of the {H}ardy potential in some {C}alder\'on-{Z}ygmund
  properties for the fractional {L}aplacian.
\newblock {\em J. Differential Equations}, 260(11):8160--8206, 2016.

\bibitem{BadTar}
Marino Badiale and Gabriella Tarantello.
\newblock A {S}obolev-{H}ardy inequality with applications to a nonlinear
  elliptic equation arising in astrophysics.
\newblock {\em Arch. Ration. Mech. Anal.}, 163(4):259--293, 2002.

\bibitem{BA2}
Albert Baernstein, II.
\newblock A unified approach to symmetrization.
\newblock In {\em Partial differential equations of elliptic type ({C}ortona,
  1992)}, Sympos. Math., XXXV, pages 47--91. Cambridge Univ. Press, Cambridge,
  1994.

\bibitem{NB}
H.~Berestycki and L.~Nirenberg.
\newblock On the method of moving planes and the sliding method.
\newblock {\em Bol. Soc. Brasil. Mat. (N.S.)}, 22(1):1--37, 1991.

\bibitem{BogDyd}
Krzysztof Bogdan and Bart\l~omiej Dyda.
\newblock The best constant in a fractional {H}ardy inequality.
\newblock {\em Math. Nachr.}, 284(5-6):629--638, 2011.

\bibitem{BrLb}
Ha\"\i~m Br\'ezis and Elliott~and Lieb.
\newblock A relation between pointwise convergence of functions and convergence
  of functionals.
\newblock {\em Proc. Amer. Math. Soc.}, 88(3):486--490, 1983.

\bibitem{CS}
Luis Caffarelli and Luis Silvestre.
\newblock An extension problem related to the fractional {L}aplacian.
\newblock {\em Comm. Partial Differential Equations}, 32(7-9):1245--1260, 2007.

\bibitem{CasFabManSan}
Daniele Castorina, Isabella Fabbri, Gianni Mancini, and Kunnath Sandeep.
\newblock Hardy-{S}obolev inequalities and hyperbolic symmetry.
\newblock {\em Atti Accad. Naz. Lincei Rend. Lincei Mat. Appl.},
  19(3):189--197, 2008.

\bibitem{NPV}
Eleonora Di~Nezza, Giampiero Palatucci, and Enrico Valdinoci.
\newblock Hitchhiker's guide to the fractional {S}obolev spaces.
\newblock {\em Bull. Sci. Math.}, 136(5):521--573, 2012.

\bibitem{DMPB}
Serena Dipierro, Luigi Montoro, Ireneo Peral, and Berardino Sciunzi.
\newblock Qualitative properties of positive solutions to nonlocal critical
  problems involving the {H}ardy-{L}eray potential.
\newblock {\em Calc. Var. Partial Differential Equations}, 55(4):Art. 99, 29,
  2016.

\bibitem{DV}
Serena Dipierro and Enrico Valdinoci.
\newblock A density property for fractional weighted {S}obolev spaces.
\newblock {\em Atti Accad. Naz. Lincei Rend. Lincei Mat. Appl.},
  26(4):397--422, 2015.

\bibitem{DLV}
Bartlomiej Dyda, Juha Lehrb\"ack, and Antti~V. V\"ah\"akangas.
\newblock Fractional {H}ardy-{S}obolev type inequalities for half space and
  {J}ohn domain.
\newblock {\em arXiv: 1709.03296v1.}, 2017.

\bibitem{FallWeth}
Mouhamed~Moustapha Fall and Tobias Weth.
\newblock Nonexistence results for a class of fractional elliptic boundary
  value problems.
\newblock {\em J. Funct. Anal.}, 263(8):2205--2227, 2012.


\bibitem{FelWan}
PATRICIO FELMER and YING WANG.
\newblock Radial symmetry of positive solutions to equations involving the
  fractional laplacian.
\newblock {\em Communications in Contemporary Mathematics}, 16(01):1350023,
  2014.



\bibitem{FrLiSi}
Rupert~L. Frank, Elliott~H. Lieb, and Robert Seiringer.
\newblock Hardy-{L}ieb-{T}hirring inequalities for fractional {S}chr\"odinger
  operators.
\newblock {\em J. Amer. Math. Soc.}, 21(4):925--950, 2008.

\bibitem{FS1}
Rupert~L. Frank and Robert Seiringer.
\newblock Non-linear ground state representations and sharp {H}ardy
  inequalities.
\newblock {\em J. Funct. Anal.}, 255(12):3407--3430, 2008.

\bibitem{FS2}
Rupert~L. Frank and Robert Seiringer.
\newblock Sharp fractional {H}ardy inequalities in half-spaces.
\newblock In {\em Around the research of {V}ladimir {M}az'ya. {I}}, volume~11
  of {\em Int. Math. Ser. (N. Y.)}, pages 161--167. Springer, New York, 2010.

\bibitem{MusGaz}
Marita Gazzini and Roberta Musina.
\newblock Hardy-{S}obolev-{M}az'ya inequalities: symmetry and breaking symmetry
  of extremal functions.
\newblock {\em Commun. Contemp. Math.}, 11(6):993--1007, 2009.

\bibitem{MusGaz1}
Marita Gazzini and Roberta Musina.
\newblock On a {S}obolev-type inequality related to the weighted
  {$p$}-{L}aplace operator.
\newblock {\em J. Math. Anal. Appl.}, 352(1):99--111, 2009.

\bibitem{GhSh}
Nassif Ghoussoub and Shaya Shakerian.
\newblock Borderline variational problems involving fractional {L}aplacians and
  critical singularities.
\newblock {\em Adv. Nonlinear Stud.}, 15(3):527--555, 2015.


\bibitem{Ili61}
V. P. Il'in.
\newblock Some integral inequalities and their applications in the theory of differentiable functions of several variables
\newblock{\em Mat. Sb. (N.S.)}, 54(96): 331--380,1961.


\bibitem{JLX}
Tianling Jin, YanYan Li, and Jingang Xiong.
\newblock On a fractional {N}irenberg problem, part {I}: blow up analysis and
  compactness of solutions.
\newblock {\em J. Eur. Math. Soc. (JEMS)}, 16(6):1111--1171, 2014.

\bibitem{Kmann}
Moritz Kassmann.
\newblock A priori estimates for integro-differential operators with measurable
  kernels.
\newblock {\em Calc. Var. Partial Differential Equations}, 34(1):1--21, 2009.





\bibitem{PLL2}
P.-L. Lions.
\newblock The concentration-compactness principle in the calculus of
  variations. {T}he locally compact case. {I}.
\newblock {\em Ann. Inst. H. Poincar\'e Anal. Non Lin\'eaire}, 1(2):109--145,
  1984.

\bibitem{PLL1}
P.-L. Lions.
\newblock The concentration-compactness principle in the calculus of
  variations. {T}he locally compact case. {II}.
\newblock {\em Ann. Inst. H. Poincar\'e Anal. Non Lin\'eaire}, 1(4):223--283,
  1984.

\bibitem{SanManFab}
G.~Mancini, I.~Fabbri, and K.~Sandeep.
\newblock Classification of solutions of a critical {H}ardy-{S}obolev operator.
\newblock {\em J. Differential Equations}, 224(2):258--276, 2006.

\bibitem{SanMan}
Gianni Mancini and Kunnath Sandeep.
\newblock On a semilinear elliptic equation in {$\Bbb H^n$}.
\newblock {\em Ann. Sc. Norm. Super. Pisa Cl. Sci. (5)}, 7(4):635--671, 2008.

\bibitem{MV}
Vladimir~G. Maz'ja.
\newblock {\em Sobolev spaces}.
\newblock Springer Series in Soviet Mathematics. Springer-Verlag, Berlin, 1985.
\newblock Translated from the Russian by T. O. Shaposhnikova.

\bibitem{musina}
Roberta Musina.
\newblock Ground state solutions of a critical problem involving cylindrical
  weights.
\newblock {\em Nonlinear Anal.}, 68(12):3972--3986, 2008.



\bibitem{MusNaz}
R.~{Musina} and A.~I. {Nazarov}.
\newblock {Fractional Hardy-Sobolev inequalities on half spaces}.
\newblock {\em ArXiv e-prints}, July 2017.


\bibitem{PP}
Giampiero Palatucci and Adriano Pisante.
\newblock Improved {S}obolev embeddings, profile decomposition, and
  concentration-compactness for fractional {S}obolev spaces.
\newblock {\em Calc. Var. Partial Differential Equations}, 50(3-4):799--829,
  2014.

\bibitem{S}
Luis Silvestre.
\newblock Regularity of the obstacle problem for a fractional power of the
  {L}aplace operator.
\newblock {\em Comm. Pure Appl. Math.}, 60(1):67--112, 2007.

\bibitem{TinTer}
A.~Tertikas and K.~Tintarev.
\newblock On existence of minimizers for the {H}ardy-{S}obolev-{M}az\cprime ya
  inequality.
\newblock {\em Ann. Mat. Pura Appl. (4)}, 186(4):645--662, 2007.

\bibitem{Tzi}
Konstantinos Tzirakis.
\newblock Sharp trace {H}ardy-{S}obolev inequalities and fractional
  {H}ardy-{S}obolev inequalities.
\newblock {\em J. Funct. Anal.}, 270(12):4513--4539, 2016.

\end{thebibliography}

\end{document}